\documentclass[12pt,reqno]{amsart}

\usepackage{graphicx}
\usepackage{amsmath}
\usepackage{amssymb}
\usepackage{amscd}
 \usepackage{tikz}
\usepackage{pdfpages}
\usepackage{bbm}
\usepackage{amsthm}
\usepackage{amsfonts}
\usepackage[utf8]{inputenc}
\usepackage[T1]{fontenc}
\usepackage{lmodern}
\usepackage[parfill]{parskip}
\usepackage{paralist}
\renewenvironment{itemize}[1]{\begin{compactitem}#1}{\end{compactitem}}
\renewenvironment{enumerate}[1]{\begin{compactenum}#1}{\end{compactenum}}

\calclayout

\newtheorem{theorem}{Theorem}
\newtheorem{theorema}{Theorem}
\newtheorem{theoremb}{Theorem}
\newtheorem{theoremc}{Theorem}
\newtheorem{theoremd}{Theorem}

\newtheorem{rk}[theorema]{Remark}
\newtheorem{lem}[theoremb]{Lemma}
\newtheorem{prop}[theoremc]{Proposition}
\newtheorem{cor}[theoremd]{Corollary}

\newcommand{\comm}[1]{}
\newcommand\1{{\bf 1}}

\newcommand\End{\op{End}}
\newcommand\Hom{\op{Hom}}
\renewcommand\a{\alpha}

\newcommand\e{\varepsilon}

\newcommand\E{{\mathcal E}}

\newcommand\g{\mathfrak{g}}
\newcommand\gl{\mathfrak{gl}}

\newcommand\h{\mathfrak{h}}

\newcommand\La{\Lambda}
\newcommand\m{\mathfrak{m}}

\newcommand\oo{\omega}

\newcommand\op[1]{\mathop{\rm #1}\nolimits}
\newcommand\ot{\otimes}
\newcommand\p{\partial}

\renewcommand\p{\mathfrak{p}}
\newcommand\R{{\mathbb R}}

\renewcommand\sl{\mathfrak{sl}}
\newcommand\so{\mathfrak{so}}
\renewcommand\sp{\mathfrak{sp}}
\newcommand\sym{\mathfrak{sym}}

\newcommand\tao{\mathfrak{o}}
\newcommand\vp{\varphi}
\newcommand\we{\wedge}

\newcommand\Z{{\mathbb Z}}

\newcommand\com[1]{}

\begin{document}

 \title[Non-Degenerate Para-Complex Structures in 6D]{Non-degenerate Para-Complex Structures \\
 in 6D with Large Symmetry Groups}
 \author{B.\,S. Kruglikov, H. Winther}
 \address{Department of Mathematics and Statistics, Faculty of Science and Technology,
 UiT the Arctic university of Norway, Troms\o\ 90-37, Norway}
 \email{ boris.kruglikov@uit.no, \quad henrik.winther@uit.no.}
 \maketitle

 \begin{abstract}
For an almost product structure $J$ on a manifold $M$ of dimension $6$ with non-degene\-ra\-te Nijenhuis tensor $N_J$, we show that the automorphism group $G=\op{Aut}(M,J)$ has dimension at most 14.
In the case of equality $G$ is the exceptional Lie group $G_2^*$.
The next possible symmetry dimension is proved to be equal to 10, and $G$ has Lie algebra $\sp(4,\R)$.
Both maximal and submaximal symmetric structures are globally homogeneous and strictly nearly para-K\"ahler.
We also demonstrate that whenever the symmetry dimension is
at least 9, then the automorphism algebra acts locally transitively.
 \end{abstract}

%\bigskip

%\bigskip

 \section{Introduction and main results}

An almost product structure $J$ on a manifold $M$ is a nontrivial endomorphism of the tangent bundle 
with $J^2=\1$, $J\neq\pm\1$. This is equivalent to a proper splitting $TM=\Delta_-\oplus\Delta_+$, 
$J|_{\Delta_\pm}=\pm\1$. 
Such structures arise in variety of applications, in particular in the study of 
double fibrations, bi-Hamiltonian structures, 
affine connections on bundles, Lagrangian 2-webs, Anosov diffeomorphisms, etc.

In this paper we study real 6-dimensional manifolds $M$ with \textit{non-degenerate} $J$, i.e.\ 
such that the Nijenhuis tensor $N_J:\La^2TM\rightarrow TM$ is an epimorphism
(6 is the minimal dimension when this is possible,
and a generic almost product structure $J$ with $\op{tr}(J)=0$ is non-degenerate).
In this case $\op{rank}\Delta_\pm=3$ (so $\op{tr}(J)=0$), and the restrictions of the Nijenhuis tensor
give the curvature tensors of the distributions $\Xi_\pm:\La^2\Delta_\pm\to\Delta_\mp$,
$X_\pm\we Y_\pm\to[X_\pm,Y_\pm]\op{mod}\,\Delta_\pm$, that are isomorphisms
(notice that if $\dim M\neq6$ the maps $\Xi_\pm$ cannot be isomorphisms simultaneously).

When the ranks of $\pm1$ eigenspaces of an almost product structure $J$ on (then necessarily even-dimensional) manifold $M$ are equal, the structure is called almost para-complex. 
Such structures with $N_J=0$ and their Hermitian and K\"ahler analogs 
(namely the triples $(g,J,\oo)$ with $g$ a Riemannian metric on $M$ and 
$\omega(X,Y)=g(X,JY)$ a 2-form with some integrability properties)
originated in \cite{R,L}, have been extensively studied in the literature
\cite{CFG,HDKN,Sc,A} and they received various physical applications \cite{AMT,GS}.

Non-degenerate almost complex structures in dimension 6 were studied in great detail in
\cite{K1,Bu,K2,AKW,KW}. To our knowledge the corresponding almost product geometry has not been addressed.
We will call it {\em non-degenerate para-complex geometry} (omitting the adjective "almost").
 % We do it in the present paper.
One might think that it should be analogous to the almost complex case,
but this is only partially true. The algebraic part of this study involves the split versions of the Lie algebras
and their representations, yet there are fewer symmetric geometries in this case.

Our first result is the maximal bound on the symmetry dimension, i.e.\ dimension of the Lie algebra
$\sym(M,J)=\{X\in\mathfrak{D}(M):L_XJ=0\}$: it is the same 14 as in the almost complex case.
However while non-degenerate almost complex geometry possesses two different maximally symmetric structures,
the maximally symmetric model in the para-complex case is unique.

 \begin{theorem}\label{Thm1}
Let $(M,J)$ be a connected non-degenerate para-complex manifold. Then $\dim\sym(M,J)\leq14$,
and in the case of equality the manifold is locally homogeneous of the type $\g_2^*/\sl_3$.
Moreover, if $\dim\op{Aut}(M,J)=14$, then $\op{Aut}(M,J)=G_2^*$ and $M$ is the globally homogenous
space $G_2^*/SL_3$, where $G_2^*$ is the algebraic
(not simply-connected) exceptional Lie group with Lie algebra $\g_2^*$.
 \end{theorem}

It is clear that $\dim\op{Aut}(M,J)\leq\dim\sym(M,J)$. Note that for a subset
of the maximally symmetric model $M\subset M_0=G_2^*/SL_3$ the symmetry algebra remains
the split exceptional Lie algebra $\g_2^*$, while the automorphism group decreases in size.
For instance, if $M=M_0\setminus\{\op{point}\}$, then $\op{Aut}(M,J)=SL_3$.

Next we are interested in the submaximal symmetry, i.e.\ such $(M,J)$ that its symmetry algebra
has the second largest dimension (accidentally in this geometry this is the same dimension
as the second largest dimension of the automorphism group).
Many geometric structures exhibit the gap phenomenon,
namely there are prohibited symmetry dimensions \cite{KT}. In the case of
non-degenerate para-complex structures the gap is four.

 \begin{theorem}\label{Thm2}
Let $(M,J)$ be a connected non-degenerate para-complex manifold that is not locally equivalent
to the maximal symmetry model of Theorem \ref{Thm1}.
Then $\dim\sym(M,J)\leq10$; in the case of equality the manifold is locally homogeneous
of the type $\sp(4,\R)/\gl_2$. Moreover, if $\dim\op{Aut}(M,J)=10$, then the connected component
$G=\op{Aut}(M,J)_0$ is either $Sp(4,\R)$ or $SO_+(2,3)$ and $M$ is globally homogenous of the type
$G/(SL_2\times\R)$.
 \end{theorem}

Let us note that $(3,6)$-distributions, i.e.\ rank 3 distributions $\Delta$ on $M^6$ with $[\Delta,\Delta]=TM$
are parabolic geometries of type $(B_3,P_3)$. They were studied by the Cartan equivalence method in \cite{Br},
and it was demonstrated that the maximal symmetry dimension is 21. In \cite{KT} it was shown that the
submaximal symmetry dimension is 11. Non-degenerate para-complex structures can be considered as
a pair of transversal $(3,6)$-distributions. We see that the maximal and submaximal dimensions drop to 14 and 10
respectively.

 \begin{rk}\label{Rk1}
The maximally symmetric model $\mathbb{S}^{3,3}=\mathbb{S}^3\times\R^3$ is
a topologically trivial 3-dimensional bundle over 3-sphere.
 % (cf.\ description \cite{Ka} of $G_2^*$ action on $\R^{3,4}$, where $\mathbb{S}^{3,3}$ is the sphere of radius $+1$).
The submaximally symmetric model $Sp(4,\R)/(SL_2\times\R)=\mathbb{S}^3\tilde{\times}V^3$ is a topologically non-trivial
3-dimensional bundle over 3-sphere. The other submaximal model is obtained by the central quotient:
$SO_+(2,3)/(SL_2\times\R)=\R P^3\tilde{\times}V^3$,
see Section \ref{S6}. % for details.
 \end{rk}

In both the maximal and submaximal symmetric models above the isotropy preserves not only the almost 
product structure $J$, but also a (unique up to scale) metric $g$ of signature $(3,3)$ such that $J^*g=-g$ 
and $\omega(X,Y)=g(X,JY)$ is a nondegenerate 2-form (almost symplectic structure). Thus we have an invariant para-Hermitian structure $(g,J,\omega)$ on $M$. This structure is defined to be nearly para-K\"ahler
(an analog of the nearly K\"ahler condition \cite{Gr}) if
 $$
\nabla^g \omega \in \Omega^3M,
 $$
and strictly nearly para-K\"ahler if, in addition, this 3-form is nonzero (this implies the Nijenhuis tensor is non-degenerate). In both highly symmetric models (with the symmetry algebra $\g_2^*$ or $\sp(4,\R)$)
this condition is satisfied. Recall that in 6D nearly-K\"ahler manifolds are Einstein with parallel Nijenhuis tensor
\cite{IZ}.

 \begin{cor}
The gap between maximal and sub-maximal symmetry dimensions of $\sym(J)$ for $\dim M =6$ is the same for non-degenerate para-complex structures as for strictly nearly para-K\"ahler structures.
 \end{cor}

Finally consider the question of transitivity of the symmetry group action.

 \begin{theorem}\label{Thm3}
Let $(M,J)$ be a connected non-degenerate para-complex manifold with symmetry algebra of dimension $>8$.
Then the structure $J$ is locally homogeneous, i.e.\ near a generic point $(M,J)$ is equivalent to a homogeneous model $G/H$, where $G$ is a Lie group of dimension $d\in\{9,10,14\}$, and $H$ its subgroup of dimension $d-6$.
 \end{theorem}

In other words, if a group of dimension $d>8$ acts on a 6-dimensional non-degenerate para-complex manifold, then
in the case $d=10,14$ it has only one orbit (global homogeneity), while for $d=9$ it has open orbits and the union
of all open orbits is dense (local homogeneity; the singular orbits can be present).

There are many examples of non-degenerate para-complex structures with the symmetry dimension 9, for instance,
parametric families on $U(1,2)/SU(1,1)$, $GL_3/SL_2$, $SU(2)^3/SU(2)_{\op{diag}}$, $SL_2^3/SL_2$, etc.
Those with semi-simple isotropy can be obtained similarly (in technique) to the almost complex case \cite{AKW}.
Note that among those listed the only compact manifold admitting a symmetric non-degenerate 
para-complex structure is $\mathbb{S}^3\times\mathbb{S}^3$, see Section \ref{S6} for details.

The rest of this paper constitutes a proof of the above theorems. Some computations in \textsc{Maple}
are available as a supplement (ancillary) to arXiv:1611.05767.

\textsc{Acknowledgements}: Both authors were partially supported by the Norwegian Research Council and 
DAAD project of Germany.

\section{1-jet determination and the possible isotropy}

We begin with the following statement that is analogous to the almost complex case \cite[Theorem 2.1(i)]{K2}.

 \begin{theorem}\label{finitedim}
The symmetry pseudogroup of a non-degenerate para-complex manifold $(M^6,J)$ is finite-dimensional.
It is 1-jet determined at any point of $M$, i.e.\ the isotropy representation is faithful everywhere.
 \end{theorem}

We will need some facts from the formal theory of differential equations, see \cite{Sp,KL} for details.
For a vector bundle $\pi:E\to M$ with the fiber $F$ let $J^k\pi$ be the space of $k$-jets of
its local sections. These spaces are equipped with the natural projections $\pi_{k,k-1}:J^k\pi\to J^{k-1}\pi$.

A system of differential equations of order $1$ is a subbundle $\E\subset J^1\pi$. Its \emph{symbol}
is the subbundle $\g_1=\op{Ker}(d\pi_{1,0}:T\E\to TJ^0)\subset T^*\ot F$, where $T=TM$.
The Spencer-Sternberg prolongations of the symbol $\g_k=\g_1^{(k-1)}$ are given
by the formula $\g_k=S^{k-1}T^*\ot\g_1\cap S^kT^*\ot F$. We also let $\g_0=F$.

Prolongations of $\E$ are defined as subsets $\E_k=\E^{(k-1)}\subset J^k\pi$,
which are zero loci of the differential corollaries of the PDEs defining $\E$ (obtained
by differentiating the defining relations by all variables $\le k-1$ times). System
$\E$ is \textit{formally integrable} (compatible) if $\E_k$ are vector subbundles of $J^k\pi$ and
$\pi_{k,k-1}:\E_k\to\E_{k-1}$ are submersions. It has finite type if eventually $\g_k=0$.
In this case the space of solutions is finite-dimensional with dimension bounded by $\sum\dim\g_k$.

 \begin{proof}
Let us consider the Lie equation on the 1-jets of infinitesimal symmetries $X\in\mathfrak{D}_M$
(space of vector fields on $M$) at various points $x\in M$ preserving the structure $J$:
 $$
\mathfrak{Lie}(J)=\{[X]^1_x:L_X(J)_x=0\}\subset J^1(TM).
 $$
Its symbol is $\bar\g_1=\sum_{\e=\pm}\Delta_\e^*\ot\Delta_\e\subset T^*\ot T$, where
$\Delta_\pm^*=\Delta_\mp^\perp\subset T^*$ for $\Delta_\pm\subset T$
and $T=TM$. This equation is formally integrable iff $J$ is integrable
($\Leftrightarrow$ $\Delta_\pm$ are integrable).
So we consider its first {\it prolongation-projection\/} $\E=\pi_{2,1}(\mathfrak{Lie}(J)^{(1)})$,
which is the Lie equation for the pair $(J,N_J)$ consisting of 1-jets of vector fields
preserving both tensors. Identifying $N_J$ with $\Xi_\pm$, the symbol of $\E$ is
 $$
\g_1=\{f\in T^*\ot T: f(\Delta_\pm)\subset\Delta_\pm,\
\Xi_\pm(f\xi,\eta)+\Xi_\pm(\xi,f\eta)=f\Xi_\pm(\xi,\eta)\,\forall\xi,\eta\in \Delta_\pm\}.
 $$
The Spencer-Sternberg prolongation $\g_1^{(1)}$ of this space equals:
 $$
\g_2=\{h\in \sum_{\e=\pm}S^2\Delta_\e^*\ot\Delta_\e:
\Xi_\pm(h(\xi,\eta),\zeta)+\Xi_\pm(\eta,h(\xi,\zeta))=h(\xi,\Xi_\pm(\eta,\zeta))\}.
 $$
Above we extend $\Xi_\pm$ to $\Lambda^2T^*\ot T$ by letting $\Xi_\pm(\xi,\eta)=0$ if either of $\xi,\eta$ belongs
to $\Delta_\mp$. Then substituting $\xi\in\Delta_+$, $\eta,\zeta\in\Delta_-$ into the defining relation
and using the fact that $\Xi_\pm$ is onto $\Delta_\mp$ we conclude vanishing of $h\in S^2T^*\ot T$, so $\g_2=0$.

Thus $\mathfrak{Lie}(J)$ has finite type and the automorphism pseudogroup $G$ of $J$ is finite-dimensional with
$\dim G=\dim\g_0+\dim\g_1<6+2\cdot9=24$. % (this bound however is not sharp).
 \end{proof}

Let us study in more detail the symbol of the equation $\E=\pi_1(\mathfrak{Lie}(J)^{(1)})$ from the preceeding proof.

 \begin{prop}\label{sl3}
The symbol of $\E$ is $\g_1\subset\sl_3\subset
\sum_{\e=\pm}\Delta_\e^*\ot\Delta_\e\subset T^*\ot T$.
 \end{prop}

 \begin{proof}
Consider the map $\Psi_+$ given by the following composition
 $$
\Delta_+\ot\La^3\Delta_+\to \Delta_+\ot\Delta_+\ot\La^2\Delta_+\to \La^2\Delta_+\ot\La^2\Delta_+
\stackrel{\Xi_+^{\ot2}}\longrightarrow\Delta_-\ot\Delta_-\to\La^2\Delta_-\stackrel{\Xi_-}\longrightarrow\Delta_+.
 $$
If $J$ is non-degenerate then $\Psi_+$ is an isomorphism and we uniquely fix volume form $\Omega_+$ on $\Delta_+^*$
by the requirement $\det\Psi_+(\cdot,\Omega_+)=1$. Similarly we get a canonical volume form $\Omega_-$ on
$\Delta_-^*$. This reduces the symbol $\bar\g_1=\gl_3\oplus\gl_3$
of the Lie equation $\mathfrak{Lie}(J)$ to $\sl_3\oplus\sl_3$.

Moreover, a combination of the volume forms and $\Xi_\pm$ gives the canonical identification
$\Delta_+=\Delta_-^*$. This further reduces $\bar\g_1$ to its diagonal subalgebra $\sl_3$,
and by the identification above we conclude the form of the isotropy representation.
 \end{proof}

 \begin{cor}\label{dim14}
If a non-degenerate para-complex manifold $(M,J)$ is connected then $\dim\sym(M,J)\leq14$.
In the case of equality the isotropy algebra $\h=\sl_3$ and the isotropy representation is $\m=V\oplus V^*$, where $V$
is the standard $\sl_3$-irrep. In general, $\h\subset\sl_3$ and the isotropy representation is the restriction of the above.
 \end{cor}

Exploiting Jordan normal forms of the isomorphism $\bar\Psi_+=\Psi_+(\cdot,\Omega_+):\Delta_+\to\Delta_+$
with $\det\bar\Psi_+=1$ ($\bar\Psi_+$ uniquely determines the analogous map $\bar\Psi_-:\Delta_-\to\Delta_-$) we get
(real) normal forms of the Nijenhuis tensors $N_J=(\Xi_+,\Xi_-)$:
 $$
\begin{pmatrix}s & 0 & 0\\ 0 & ts^{-1} & 0\\ 0 & 0 & t^{-1}\end{pmatrix},\
\begin{pmatrix}s^{-2} & 0 & 0\\ 0 & s\cos t & s\sin t\\ 0 & -s\sin t & s\cos t\end{pmatrix},\
\begin{pmatrix}s^{-2} & 0 & 0\\ 0 & s & 1\\ 0 & 0 & s\end{pmatrix},\
\begin{pmatrix}1 & 1 & 0\\ 0 & 1 & 1\\ 0 & 0 & 1\end{pmatrix}.
 $$
We see that the number of essential parameters (moduli) is 2, in exact correspondence with
the normal forms of the non-degenerate Nijenhuis tensors of almost complex structures in 6D from \cite{K1}.
Thus $\g_1$ is either $\sl_3$ or $\gl_2$, a 4-dimensional solvable Lie algebra or
a 2-dimensional Lie algebra. The only fact that we need though is the inclusion $\h\subset\sl_3$
from Corollary \ref{dim14}.

 \section{Lie algebra Extensions of $\h$-Modules}\label{S3}

% [In this section we repeat our usual theory of deformation and reconstruction. It is essentially the same as in KW. We could possibly improve this. Use of the Jacobi identity in proving Lemma \ref{cohomologylemma} (which is only pertaining to modules) is a small problem.]

 \subsection{The $\h$-module structure of $\g$}\label{S31}
In the event that $\g$ does not split into a direct sum of $\h$ and $\m$,
we choose an arbitrary complement of $\h$ which we will still denote by $\m$, even though it is not a submodule.
We have
 \begin{align*}
[h,m]=\vp(h)m+h\,m\in \h \oplus \m
 \end{align*}
for some $\vp:\h\rightarrow \m^\ast \otimes \h$. Here $h\,m$ denotes the action of $\h$ on the module $\m=\g/\h$. Let us change the complement $\m$ by some operator $A:\m\rightarrow \h$, so that the new complement is $\m_{new} = \{(A\,m,m)\,|\,m\in\m \}$. Then
 \begin{align*}
[h,Am+m]=(\vp(h)m+[h,Am]-A\,h\,m)+(A\,h\,m+h\,m)\in \h \oplus \m
 \end{align*}
and the first three terms describe $\vp_{new}$. Denoting by $d_\h$ the Lie algebra differential in the complex $\Lambda^\bullet\h^\ast \otimes \m^\ast \otimes \h$ of $\Hom(\m,\h)$-valued forms on $\h$, this equals
 \begin{align*}
\vp_{new}=\vp+d_\h A.
 \end{align*}
Moreover, from the Jacobi identity between elements $m,h_1,h_2$ we get $d_\h \vp =0$, so $\vp$ is a cocycle.
This gives the following statement (it can also be seen as a result of the isomorphism $\op{Ext}^1_\mathfrak{h}(\mathfrak{m},\mathfrak{h})=H^1(\mathfrak{h},\op{Hom}(\mathfrak{m},\mathfrak{h}))$ and the extension obstruction for modules \cite{Gi}).

 \begin{lem}\label{cohomologylemma}
The equivalence classes of $\h$-modules $\g$ with $\g /\h \simeq \m$ are given by the Lie algebra cohomology $H^1(\h,\Hom(\m,\h))$. In particular, if this cohomology vanishes, then $\g=\h\oplus\m$ is a direct sum.
 \end{lem}

\subsection{Lie algebra structures on the $\h$-module $\g$}\label{S32}
Let $\h$ be a Lie algebra and $\g$ be an $\h$-module such that $\h\subset\g$ as a submodule.
By a \textit{Lie algebra extension} of $\h$ on $\g$, we mean a bracket operation
 \begin{align*}
[,]:\La^2\g\rightarrow\g
 \end{align*}
which satisfies the usual Lie algebra axioms and the restriction criteria that
 \begin{align*}
&[,]:\La^2\h\rightarrow\h\\
&[,]:\h\wedge \g\rightarrow\g
 \end{align*}
are respectively the Lie bracket of $\h$ and the module action of $\h$ on $\g$. Specialize to the case described in the previous subsection,  $\g/\h=\m$. We introduce two operations on the cohomology representative $\vp$.

Let  $\delta: \h^*\otimes\m^*\otimes\h \rightarrow \h^*\ot\La^2\m^*\ot\m$ be given by
$$
\delta\vp(h)(u_1,u_2)=\vp(h,u_1)\cdot u_2-\vp(h,u_2)\cdot u_1.
$$
Given $\theta_\m\in\La^2\m^*\otimes\m$ with $\delta\vp=d\theta_\m$, define the operator
$$
Q:\h^*\otimes\m^*\ot\h\to\h^*\otimes\La^2\m^*\otimes\h
$$
by
$$
Q\vp(h)(u_1,u_2)=\vp(\vp(h,u_1),u_2)-\vp(\vp(h,u_2),u_1)-\vp(h,\theta_\m(u_1,u_2)).
$$
Let us also define the linear operators $q:\m^*\otimes\h\to\h^*\otimes\La^2\m^*\otimes\h$, $\sigma\mapsto q_\sigma$, and
$p:(\La^2\m^*\otimes\m)^\h\to\h^*\otimes\La^2\m^*\otimes\h$, $\nu\mapsto p_\nu$, by the formulae
 \begin{align*}
q_\sigma(h)(u_1,u_2)=&
\,d\sigma(\vp(h,u_1),u_2)-d\sigma(\vp(h,u_2),u_1)+\vp(d\sigma(h,u_1),u_2)\\
&-\vp(d\sigma(h,u_2),u_1)+d\sigma(d\sigma(h,u_1),u_2)-d\sigma(d\sigma(h,u_2),u_1)\\
&-\vp(h,\delta\sigma(u_1,u_2))-d\sigma(h,\theta_\m(u_1,u_2))-d\sigma(h,\delta\sigma(u_1,u_2));\\
p_\nu(h)(u_1,u_2)=&\,\vp(h,\nu(u_1,u_2)),
 \end{align*}
and also denote $\Pi_\vp=\op{Im}(p_\nu)\,\op{mod}B^1(\h,\La^2\m^*\ot\h)\subset H^1(\h,\La^2\m^*\ot\h)$. Then we have the following result for the proof of which we refer to \cite{note}.

 \begin{theorem}\label{newcohomstatement}
The Jacobi identity $\op{Jac}(v_1,v_2,v_3)=0$ with 1 argument from $\h$ and
the other from $\m$ constrains the cohomology $[\vp]\in H^1(\h,\m^*\ot\h)$ so:\\
(1) $[\delta\vp]=0\in H^1(\h,\La^2\m^*\ot\m)$,  whence $\delta\vp=d\theta_\m$;\\
(2) $[Q\vp]\equiv0\in H^1(\h,\La^2\m^*\ot\h)\,\op{mod}\Pi_\vp$, so $Q\vp=d\theta_\h$ for some choices of $\vp,\theta_\m$.
 \end{theorem}

Note that if $\h$ is semi-simple then $H^1(\h,\mathbb{V})=0$ for any $\h$-module $\mathbb{V}$,
so choosing $\vp=0$, the solutions to the above constraints are equivariant $\theta_\m$, $\theta_\h$,
yielding the last ingredient of the reconstruction $\theta=\theta_{\mathfrak{h}}+\theta_{\mathfrak{m}}:\Lambda^2\mathfrak{m}\to\mathfrak{g}=\mathfrak{h}\oplus\mathfrak{m}$ that is subject to the Jacobi identity with all three arguments are from $\mathfrak{m}$..

\section{Maximally symmetric model}\label{S4}

Let $\h=\sl_3$ and $\m=V\oplus V^\ast$ be as in Corollary \ref{dim14}. 
Since $\h$ is semi-simple, all its modules
are completely reducible, so we have $\g=\h\oplus\m$ as an $\h$-module. 
We will classify the Lie algebra extensions
of $\h$ on $\g$ by applying the results from Section \ref{S32}, 
and this classification forms the first step of proving Theorem \ref{Thm1}.

 \subsection{Reconstruction of the Lie algebra}
The $\h$-invariant decomposition
 $$
\Lambda^2\m=\R\oplus V\oplus V^*\oplus \h
 $$
gives the space of equivariant maps $\Lambda^2\m\rightarrow \g$. It is identified with the space of invariant brackets
$\mathfrak{B}(\h,\g)$ and decomposes into horizontal and vertical parts
 $$
\mathfrak{B}(\h,\g)=(\Lambda^2\m^\ast\otimes\m)^\h \oplus (\Lambda^2\m^\ast\otimes\h)^\h,
 $$
The dimension of the spaces of horizontal and vertical brackets are 2 and 1, respectively.
The horizontal bracket is given by two maps $\Lambda^2 V^\ast\rightarrow V$ and $\Lambda^2 V\rightarrow V^\ast$,
that are contractions with $\h$-invariant volume forms $\omega^\ast$ on $V^*$ and $\omega$ on $V$,
and $V^*\ot V\to0$. The vertical bracket is given by $\Lambda^2 V\to0$, $\Lambda^2 V^*\to0$, and
$V^*\ot V\ni\theta\ot v\mapsto \op{Tr}_0(\theta\ot v):=\theta\ot v-\frac13\theta(v)\1_V\in\sl(V)$.

Consider now the Jacobi identities on $\m$. Let $v_1,v_2,v_3\in V$ and $\theta_1, \theta_2,\theta_3 \in V^\ast$
be a basis and its dual co-basis; let
$\omega^\ast=\alpha_1\, \theta_1 \wedge \theta_2 \wedge \theta_3$ and
$\omega=\alpha_2\, v_1 \wedge v_2 \wedge v_3$. We rescale the vertical bracket by the parameter $\beta$ and then compute
 \begin{align*}
\op{Jac}(v_1,v_2,v_3)&=[\omega^*(v_1,v_2),v_3]+[\omega^*(v_3,v_1),v_2]+[\omega^*(v_2,v_3),v_1]\\
&=\alpha_1\sum[\theta_i,v_i]=\alpha_1 \beta\, \text{Tr}_0(\1_V)=0,
 \end{align*}
and similarly get $\op{Jac}(\theta_1,\theta_2,\theta_3)=0$. Next we compute
 \begin{align*}
\op{Jac}(v_1,v_2,\theta_3)&=[\omega^*(v_1,v_2),\theta_3]+[[\theta_3,v_1],v_2]+[[v_2,\theta_3],v_1]\\
&=\alpha_1[\theta_3,\theta_3]+\beta(\theta_3(v_2)\,v_1-\theta_3(v_1)\,v_2)=0,
 \end{align*}
and similarly get $\op{Jac}(v_i,v_j,\theta_k)=0$, $\op{Jac}(v_i,\theta_j,\theta_k)=0$ whenever
the indices $i,j,k$ are distinct. Finally the identity
 \begin{multline*}
\op{Jac}(v_1,v_2,\theta_1)=[\omega^*(v_1,v_2),\theta_1]+[[\theta_1,v_1],v_2]+[[v_2,\theta_1],v_1]\\
=\alpha_1[\theta_3,\theta_1]+
\beta(\theta_1(v_2)\,v_1-\tfrac13\theta_1(v_1)v_2-\theta_1(v_1)\,v_2)
=\alpha_1\alpha_2 v_2 -\tfrac43\beta v_2 = 0
 \end{multline*}
yields the equation $\beta=\frac34\alpha_1\alpha_2$. The same equation arises from all the identities $\op{Jac}(v_i,v_j,\theta_k)=0$, $\op{Jac}(\theta_i,\theta_j,v_k)=0$ where $k=i\vee j$. These are all the Jacobi identities, yielding three families of solutions.

The first two are $\beta=\alpha_1=0$ and $\beta=\alpha_2=0$. In these two cases, $\m$ is realized as a two-step nilpotent ideal in $\g$. The image of the Nijenhuis tensor is contained in the commutator subalgebra of $\m$,
whence $N_J$ is degenerate.

The last solution is $\beta\not=0$ can be normalized $\alpha_1=\alpha_2=2$, $\beta=3$.
Then it is easy to see that $\g$ is isomorphic to $\g_2^*$, as was claimed.

 \subsection{Global homogeneity}\label{S42}
Let us demonstrate that a non-degenerate para-complex manifold $(M,J)$ with the symmetry $\g_2^*$
has no singular orbits.

Suggesting the opposite, let $G$ be (even local) symmetry group with Lie algebra
$\g_2^*$ (in the next subsection we show $G=G_2^*$).
The singular orbit $O=G\cdot o=G/H$ has the isotropy algebra
$\h=\op{Lie}(H)\subset\mathfrak{sl}_3$ at a non-singular point by Theorem \ref{finitedim} 
and Corollary \ref{dim14}.
Since $\dim O<6$, we get $\dim\h>14-6=8=\dim\mathfrak{sl}_3$, which is impossible, unless $\mathfrak{h}=\mathfrak{g}_2^*$ and $O=o$ is a singular point. In the latter case, by the Thurston stability theorem \cite{T}, the action is linearizable near $o$, and so $\mathfrak{g}_2^*$ has a faithful representation on $T_oM$. But $\mathfrak{g}_2^*$ has the smallest non-trivial linear representaiton in dimension 7, and this contradiction proves the claim.

Thus $M$, whenever connected, is the unique orbit of the Lie group $G$ action, and so is globally homogeneous.

 \subsection{Uniqueness of the maximally symmetric model}

Let $\dim\op{Aut}(M,J)=14$. Then $M$ is a globally homogeneous space. One such choice is given by $M_0=G_2^*/SL_3$.
Since this has homotopy type of 3-sphere, it is simply-connected, $\pi_1(M_0)=0$, and moreover $\pi_2(M_0)=0$.

The group $G_2^*$ does not have a center, and its double cover $\widetilde G_2^*$ is simply-connected \cite{Ka}.
We claim that preimage of $SL_3$ in this double-cover is the universal cover $\widetilde{SL}_3$
(recall that $\pi_1(SL_3)=\Z_2$). Indeed, from the exact homotopy sequence of the fibration giving $M_0$
 $$
\dots\to\pi_2(M_0)\to\pi_1(SL_3)\to\pi_1(G_2^*)\to\pi_1(M_0)\to\cdots
 $$
we conclude that a generator of the fundamental group of $SL_3$ is also a generator for that of $G_2^*$,
and this implies the claim.

Thus $\widetilde G_2^*/\widetilde{SL}_3=M_0$ and we proved this is the only maximally symmetric model
with the automorphism group of dimension 14.

\section{Submaximally symmetric model}

In this section we obtain the homogeneous models from Theorem \ref{Thm2}.

\subsection{Subalgebras of $\sl_3$}\label{S51}
By Mostow's theorem a proper maximal subalgebra of a semi-simple Lie algebra
is either parabolic or semi-simple or the stabilizer of a pseudo-torus \cite{M}.

% We begin by checking the dimension of the possibilities for $\h$.

The pseudo-tori of $\sl_3$ are the Lie algebras $\mathfrak{t}$ of circle-subgroups in $SO(3)$, which are all equivalent under conjugation, and have stabilizer $\mathfrak{t}\oplus \R$ of dimension 2. The semi-simple subalgebras of $\sl_3$ are $\sl_2$ and $\so_3$, both of dimension 3.

There are, up to conjugation, two maximal parabolic subalgebras $\p_1$ and $\p_2$, both of dimension 6.
These are equivalent under an outer automorphism of $\sl_3$, and without loss of generality we restrict
to $\p_1$, which is the stabilizer of a line in $\R^3$. Hence we will consider the subalgebras of $\p_1$.

As an abstract Lie algebra, $\p_1=\gl_2\ltimes \R^2$. Up to conjugation, it has two maximal 5-dimensional subalgebras. The first is $\p_{12}=(\R z \oplus \mathfrak{b}^2)\ltimes \R^2$, the Borel subalgebra of $\sl_3$, where $\mathfrak{b}^2$ is a Borel subalgebra of $\sl_2 \subset \gl_2$ and $z$ is the grading element of $\p_1$, it generates the center of $\gl_2$. The second subalgebra is $\sl_2\ltimes \R^2\subset\p_1$.

There are two conjugacy classes of maximal 4-dimensional subalgebras of $\p_1$. These are the classes of $\gl_2$, and of $(\R z \oplus \R t)\ltimes \R^2$, where $t\in\sl_2$ has negative Killing norm. The other 4-dimensional subalgebras of $\p_1$ (up to conjugation) are then codimension 1 subalgebras of $\p_{12}$ or $\sl_2\ltimes \R^2$. In fact, all of these will be subalgebras of $\p_{12}$, because they must be solvable and contain no element of negative Killing norm, and such subalgebras of $\sl_3$ are conjugate to subalgebras of the Borel subalgebra.

A 4-dimensional subalgebra of the 5-dimensional $\p_{12}$ must have at least a 2-dimensional intersection with the 3-dimensional subalgebra $\R z \oplus \mathfrak{b}^2$. This intersection is a subalgebra. The first possibility is that the intersection is the whole $\R z \oplus \mathfrak{b}^2$. This preserves a unique 1-dimensional subalgebra $\R$ of the ideal $\R^2$, and hence the 4D subalgebra must be $(\R z \oplus \mathfrak{b}^2)\ltimes \R$ in this case.

If the intersection is 2-dimensional, then it can be either Abelian or non-Abelian. If Abelian, it is of the form $(\R z \oplus \R t)$ for $t\in\mathfrak{b}^2$, and there are two conjugacy classes, determined by whether $t$ has positive or null Killing norm. The 4-dimensional subalgebras which realize this are of the form $(\R z \oplus \R t)\ltimes \R^2$.

There is a 1-dimensional family of 2-dimensional solvable subalgebras $\mathfrak{s}^2\subset\R z \oplus \mathfrak{b}^2$
not conjugate to each other: $\mathfrak{s}^2=(\R(h+l\,z))\ltimes (\R e)$, where $l\in\R$ is the essential parameter,
and $e,h\in \mathfrak{b}^2$ with $[h,e]=e$. These realize the 4-dimensional subalgebras $\mathfrak{s}^2\ltimes \R^2$
that also are pairwise non-conjugate.

%These can be further classified by considering the decomposition into reductive Levi factor (which is abelian for solvable algebras) and nilradical.
%
%By the Engels theorem, the nilradical can be conjugated into the nilradical of $\p_{12}$. This has dimension 3, and is $(\mathfrak{b}^2)^{(1)}\ltimes \R^2\simeq \mathfrak{heis}(3)$, where $(\mathfrak{b}^2)^{(1)}$ means the derived subalgebra of $\mathfrak{b}^2$ and $\mathfrak{heis}(3)$ is the 3-dimensional Heisenberg algebra. Thus we may look at subalgebras of $\mathfrak{heis}(3)$ of codimension at most 2 (since we are looking for subalgebras of $\p_1$ of dimension 4).

We summarize the information about subalgebras $\h$ of $\sl_3$, considered up to outer automorphism, with $\dim\h\ge 4$ in the following table.

%\begin{prop}
%Up to conjugation in $\g$, the subalgebras $\h$ of a parabolic subalgebra $\p\subset\g$ for $\g=\sp(1,1)$ or $\g=\sp(4,\R)$ with $\dim \h \ge 5$ are graded (in the inclusion given below) and are the following:
%\end{prop}

\begin{center}
\begin{tabular}[t]{| l | l | l | l | l |}
    \hline
    $\dim\h$    &   $\h$    & Notes\\  \hline
      8  &   $\sl_3$   & non-proper\\  \hline
      6  &   $\p_1$, $\p_2$   & has a non-trivial Levi factor\\ \hline
      5  &   $\p_{12}$   &\\  \hline
      5  &   $\sl_2\ltimes \R^2$  & has a non-trivial Levi factor\\  \hline
      4  &   $\gl_2$   & has a non-trivial Levi factor\\  \hline
      4  &   $(\R z \oplus \R t)\ltimes \R^2$   & $||t||<0$\\  \hline
      4  &   $(\R z \oplus \mathfrak{b}^2)\ltimes \R$   &\\  \hline
      4  &   $(\R z \oplus \R t)\ltimes \R^2$   & $||t||=0$\\  \hline
      4  &   $(\R z \oplus \R t)\ltimes \R^2$   & $||t||>0$\\  \hline
      4  &   $\mathfrak{s}^2\ltimes \R^2$   & depends on a parameter $l$\\  \hline
\end{tabular}
\end{center}

\subsection{Cohomology of subalgebras of $\sl_3$}\label{S52}
In this section we compute the equivalence classes of $\h$-modules $\g$ with $\g/\h=\m$, where $\m$ is the module corresponding to the restriction of the representation $V\oplus V^*$ of $\sl_3$ from Corollary~\ref{dim14} to $\h$. This means classifying representations $\phi$ such that the following diagram of non-trivial Lie algebra homomorphisms commute
and $\phi$ induces the adjoint action on $\h$.
\com{
\begin{center}\begin{tikzpicture}
  \matrix (m) [matrix of math nodes,row sep=3em,column sep=4em,minimum width=2em]
  {\h & \sl_3 \\ \text{Stab}(\h,\g) & \End(\m) \\};
  \path[-stealth]
    (m-1-1) edge node [right] {} (m-1-2)
	(m-1-1) edge node [right] {$\phi$} (m-2-1)
	(m-1-2) edge node [right] {} (m-2-2)
	(m-2-1) edge node [right] {} (m-2-2);
\end{tikzpicture}\end{center}
}
$$\begin{CD}
\h @>>> \sl_3\\
@VV\phi V @VVV\\
\text{Stab}(\h,\g) @>>> \End(\m)
\end{CD}$$
Here  $\text{Stab}(\h,\g)$ is the space of maps $\g \rightarrow \g$ for which the subspace $\h$ is invariant (upper block triangular in the vector space decomposition $\g=\h \oplus \m$). In the bottom row, $\g$ and $\h$ should be considered as vector spaces.

The representation matrices of $\phi$ are then given by choosing cocycle representatives of cohomology as described in Lemma \ref{cohomologylemma}. These representatives are elements $\vp\in\h^*\otimes\m^* \otimes \h$. The contraction $\vp(x)$ with $x\in\h$ gives the strictly upper diagonal block of the representation matrix of $x$ corresponding to $\phi$. The diagonal blocks are given by the action on $\h$ and $\m$, and does not depend on $\vp$.

When the cohomology is 1-dimensional, we
% may permit rescaling $\vp$ to reduce it to the discrete cases $\vp=0$ and $\vp\not=0$.
% This is because the diagonal matrix with $1$ in the $\h$-block entries and $\lambda$ in the $\m$-block entries commutes with the
% diagonal blocks of the $\h$-action, and so will not affect the presence of solutions to the Jacobi identity.
can rescale the $\m$-component to achieve $[\vp]=0$ or $[\vp]=1$.
Computation of the cohomology was performed in the \textsl{DifferentialGeometry} package of \textsc{Maple}.

\begin{prop}\label{nocohomsubalgs}
For the subalgebras $\h=\p_1$, $\h=\p_{12}$, $\h=(\R z \oplus \mathfrak{b}^2)\ltimes \R$, $\h=\gl_2$ and all subalgebras of the form $\h=(\R z \oplus \R t)\ltimes \R^2$ of $\sl_3$, we have $\text{H}^1(\h,\Hom(\m,\h))=0$.
\end{prop}

\begin{prop}\label{sl2r2cohom}
For the subalgebra $\sl_2\ltimes\R^2$ we have $\dim \text{H}^1(\h,\Hom(\m,\h))=1$.
\end{prop}

This gives the cohomology for all cases with $\dim \h\ge 4$, except for those of the form $\h=\mathfrak{s}^2\ltimes \R^2$.
These, defined in \S\ref{S51}, depend on a parameter $l\in\R$.

 \begin{prop}\label{magicnumbers}
Let $\h=\mathfrak{s}^2\ltimes \R^2 \subset \sl_3$ be as above. Then $\text{H}^1(\h,\Hom(\m,\h))=0$, unless $l \in \{\frac{9}{2}, 3,\frac{3}{2},\frac{9}{10},\frac{3}{4},0,\frac{-3}{10},\frac{-3}{4},\frac{-3}{2}\}$.
We have $\dim \text{H}^1(\h,\Hom(\m,\h))=1$ for all exceptional $l$ save for $l=\frac{3}{2}$, in which case $\dim \text{H}^1(\h,\Hom(\m,\h))=\nolinebreak 6$.
 \end{prop}
Now we apply Theorem \ref{newcohomstatement} to conclude that the majority of these non-trivial cohomologies
do not correspond to modules admitting Lie algebra extensions: % (\S\ref{S32}):
 \begin{prop}\label{nontrivial4dcohom}
Let $\h=\mathfrak{s}^2\ltimes \R^2$ be as above and let $[\vp]\in \text{H}^1(\h,\Hom(\m,\h))$.
If $l\not=\frac{3}{2}$, then $[\delta\vp]=0$ if and only if $[\vp]=0$.
 \end{prop}

\subsection{Inducing the Nijenhuis tensor}\label{S53}
In this section, we solve the equations from Theorem \ref{newcohomstatement} to parametrize possible Lie algebra
structures on $\g$. In the case $\h$-module $\g$ splits this reduces to computing the space $\mathfrak{B}(\h,\g)$
of $\h$-equivariant brackets (see Sections \ref{S3}-\ref{S4}). Then we solve the remaining equations from the Jacobi identity
and check whether the invariant almost product structures on $\m$ are non-degenerate.

Note that whenever the decomposition $\g=\h\oplus\m$ is $\h$-invariant, the space of $\h$-equivariant brackets is at least 2-dimensional, because it contains the space of $\sl_3$-invariant horizontal brackets.
 % (the generators correspond to $\Xi_\pm$)
These were already considered in Section~\ref{S4}, where we showed that without a vertical bracket,
the Lie subalgebra $\m$ is nilpotent
 % (only one of the brackets $\Xi_\pm$ is used)
and the Nijenhuis tensor degenerates.

We begin with the subalgebras of $\sl_3$ from Proposition \ref{nocohomsubalgs}, in which case the module $\g=\h\oplus\m$ splits.
 \begin{prop}
The subalgebras $\h=\p_1$, $\h=\p_{12}$, $\h=(\R z \oplus \mathfrak{b}^2)\ltimes \R$ and all subalgebras of the form $\h=(\R z \oplus \R t)\ltimes \R^2$ satisfy $\mathfrak{B}(\h,\g)=(\Lambda^2 \m^\ast \otimes \m)^{\sl_3}$, i.e.\
$\dim\mathfrak{B}(\h,\g)=2$ and there are no additional equivariant brackets.
 \end{prop}

The exception is $\h=\gl_2$.
 \com{
\begin{prop}\label{gl2givessubmaximal}
The subalgebras $\h=\gl_2$ has a 7-dimensional space of equivariant brackets, of which 4 are horizontal and 3 are vertical. There is a family of solutions to the Jacobi identity for which the invariant almost product structure is non-degenerate. For all such solutions $\g\simeq \sp(4,\R)$.
\end{prop}
 }
\begin{prop}\label{gl2givessubmaximal}
For the subalgebra $\h=\gl_2\subset\sl_3$ we have $\dim\mathfrak{B}(\h,\g)=7$, 
and there are 4 horizontal and 3 vertical
equivariant brackets. There is a family of solutions to the Jacobi identity for which the invariant almost product structure is non-degenerate. For all such solutions $\g\simeq \sp(4,\R)$.
\end{prop}

\begin{proof}
The $\sl_3$-invariant decomposition $\m = V\oplus V^\ast$ can be further decomposed with respect to $\gl_2$. Let's write $W$ for the standard $\sl_2$-module, $S^k W$ for the irreducible $\sl_2$-module of dimension $k+1$, and $S^k W(\lambda)$ for the irreducible $(k+1)$-dimensional $\gl_2$-module with $\lambda$ being the weight of the center (3 times the eigenvalue of the grading element $z$).
We decompose the $\gl_2$-modules so
 \begin{gather*}
\gl_2=S^2W(0)\oplus \R(0),\\
V =W(1)\oplus \R(-2),\qquad V^\ast= W(-1)\oplus \R(2).
 \end{gather*}
Now $\Lambda^2\m=\Lambda^2 V \oplus V \otimes V^\ast \oplus \Lambda^2 V^\ast$ and because $\Lambda^2W=\R$ we get
\begin{align*}
&\Lambda^2 V=\R(2)\oplus W(-1)\\
&V\otimes V^\ast= S^2W(0)\oplus \R(0) \oplus W(3)\oplus W(-3) \oplus \R(0)\\
&\Lambda^2 V^\ast=\R(-2)\oplus W(1)
\end{align*}
Except for $W(3)$ and $W(-3)$, all these submodules can be found in $\g=\gl_2\oplus \m$, and hence contribute linearly independent equivariant maps $\Lambda^2\m\rightarrow \g$. These span the 7-dimensional space of equivariant brackets. We note that the vertical brackets all arise from the term $V\otimes V^\ast$, while the horizontal brackets come from $\Lambda^2 V$ and $\Lambda^2 V^\ast$.

We parametrize the brackets by defining a basis of $\g$. Let $s=3z$ be 3-times the grading element of $\sl_3$,
and let $h,e,f$ be a standard basis of $\sl_2$, i.e. $[h,e]=2e, [h,f]=-2f, [e,f]=h$. Let $v_1,v_2$ be a standard basis of $W(1)$,
i.e. eigenvectors of $h$ with eigenvalues $1$ and $-1$, respectively, and of $s$ with eigenvalue $1$.
Let $r$ be a basis of $\R(-2)$. Define $\theta_1,\theta_2,\varsigma$ to be the dual basis of $v_1,v_2,r$.
Then $s,h,e,f,v_1,v_2,r,\theta_1,\theta_2,\varsigma$ is a basis of $\g$ and
the equivariant brackets on $\m$ are given by the formula
\begin{align*}
&[v_1, v_2] = a_1 \varsigma, [v_1, r] = -a_3 \theta_2, [v_2, r] = a_3 \theta_1;\quad [r, \varsigma] = b_1 s,\\
&[v_1, \theta_1] = -b_2 h+b_3 s, [v_1, \theta_2] = -2 b_2 e, [v_2, \theta_1] = -2 b_2 f, [v_2, \theta_2] = b_2 h+b_3 s;\\
&[\theta_1, \theta_2] = a_2 r, [\theta_1, \varsigma] = -a_4 v_2, [\theta_2, \varsigma] = a_4 v_1
\end{align*}
with parameters $a_1,a_2,a_3,a_4\in\R$ for the horizontal brackets and $b_1,b_2,b_3\in\R$ for the vertical ones
(as usual we omit the trivial brackets).

Computing the Jacobi identities of these brackets yields three families of solutions. The first two correspond to nilpotent Lie algebra structures on $\m$, and are given by either all parameters zero except $a_1,a_3$, or all parameters zero except $a_2,a_4$. In both  cases, one of the distributions $\Delta_+=V$ or $\Delta_-=V^*$ has vanishing curvature, so the Nijenhuis tensor is degenerate.

The last solution is given by $a_1 = \tfrac{a_3 a_2}{a_4}, b_1 = a_3 a_4, b_2 = \tfrac{1}{2} a_3 a_2, b_3 = -\tfrac{1}{2} a_3 a_2$.
If $a_i=0$ for some $i=1,2,3,4$, the Nijenhuis tensor degenerates.
Thus assume $a_i\neq0$, $1\leq i\leq 4$. Then the Lie algebra $\g$ is semi-simple, and hence simple due to dimension.
The signature of its Killing form is $(6,4)$ independently of the parameters, whence $\g\simeq\sp(4,\R)$.
In fact, all these parameters are equivalent by an inner automorphism.
The distributions $V$ and $V^\ast$ have non-degenerate curvatures, resulting in a non-degenerate para-complex structure $J$.
\end{proof}

Now we consider the subalgebra $\h=\sl_2\ltimes\R^2$ from Proposition \ref{sl2r2cohom}.
The cohomology is 1-dimensional, so we distinguish the two cases $[\vp]=0$ and $[\vp]\not=0$.

 \begin{prop}
Let $\h=\sl_2\ltimes\R^2$ and $[\vp]=0$. Then $\dim\mathfrak{B}(\h,\g)=9$, and there are
7 horizontal and 2 vertical equivariant brackets. These yield the following possible structure equations for $\g$
(without Jacobi identity yet):
{\small\begin{align*}
& [v_1, v_2] =  a_1 w_3, [v_1, v_3] = - a_1 w_2, [v_2, v_3] =  a_4 v_1+ a_1 w_1,
  [v_2, w_2] =  a_6 v_1, [w_2, w_3] =  a_2 v_1,\\
& [v_1, w_1] = ( a_7+ a_6) v_1, [v_2, w_1] =  b_1 x_2+ a_3 w_3+ a_7 v_2,
  [v_3, w_1] = -b_1 x_1- a_3 w_2+ a_7 v_3,\\
& [v_3, w_3] =  a_6 v_1, [w_1, w_2] = b_2 x_1+ a_5 w_2+ a_2 v_3, [w_1, w_3] =  b_2 x_2+ a_5 w_3- a_2 v_2,\\
& [x_1, v_2] = v_1, [x_1, w_1] = -w_2, [x_2, v_3] = v_1, [x_2, w_1] = -w_3, [e, v_3] = v_2, [e, w_2] = -w_3, \\
& [f, v_2] = v_3, [f, w_3] = -w_2, [h, v_2] = v_2, [h, v_3] = -v_3,
[h, w_2] = -w_2, [h, w_3] = w_3,\\
&[x_1, e] = x_2, [x_1, h] = x_1, [x_2, f] = x_1, [x_2, h] = -x_2, [e, f] = h,
[e, h] = -2 e, [f, h] = 2 f.
\end{align*}}\!\!
Here $e,f,h,x_1,x_2$ form a basis of $\h$ and $v_1,v_2,v_3,w_1,w_2,w_3$ a basis of $\m$. If the Jacobi identity is satisfied for $\g$, then the Nijenhuis tensor is degenerate.
 \end{prop}

 \begin{proof}
First, note that the brackets on the subspace $R=\R^2\oplus \m$ must be equivariant with respect to $\sl_2$, since $R$ is invariant. This decomposes as $R=W_0\oplus W_+\oplus \R_+ \oplus W_-\oplus \R_-$, with $V=W_+\oplus \R_+$ and $V^\ast=W_-\oplus \R_-$, as an $\sl_2$-module (the indices indicate where the parts belong to, but $W_0=W_+=W_-$ and $\R_+=\R_-$ as $\sl_2$-modules). We have
\begin{align*}
&\Lambda^2 V=\R\oplus W\\
&V\otimes V^\ast= S^2W\oplus \R \oplus W\oplus W \oplus \R\\
&\Lambda^2 V^\ast=\R\oplus W
\end{align*}
with respect to $\sl_2$, which gives a space of $\sl_2$-equivariant brackets of dimension 21. One may then compute the subspace which is also equivariant with respect to $\R^2$, which has dimension 9 and consists of the brackets given above.

Next, note that if $a_1=0$, then $V$ is involutive, and if $a_2=0$, then $V^\ast$ is involutive. However, we have the Jacobi identity
$$
\text{Jac}(v_2,w_2,w_3)= a_1 a_2 w_3
$$
Hence $a_1 a_2=0$ and so at least one distribution is involutive, and the Nijenhuis tensor of the associated almost product structure is degenerate.
 \end{proof}

 \begin{prop}
Let $\h=\sl_2\ltimes\R^2\subset \sl_3$ and $[\vp]\not=0$. Then $\g=\sl_3\ltimes V$, where $V$ is the standard $\sl(3)$-module, and the inclusion $i:\h\rightarrow \g$ is equivalent to the composition of $k:\sl_3\rightarrow \sl_3\ltimes V$ and $j:\h\rightarrow \sl_3$, where $j,k$ are the obvious subalgebra inclusions. The Nijenhuis tensor of the almost product structure associated with this solution is degenerate.
 \end{prop}

 \begin{proof}
We may assume that the complement $\m$ to $\h$ in the $\h$-module $\g$ is $\sl_2$-invariant, because modules of
semi-simple Lie algebras are completely reducible, and $\R^2$ is an $\sl_2$-submodule.
Hence the cochain representative $\vp$ of $[\vp]$ vanishes on $\sl_2$.
Let $e,f,h,x_1,x_2$ be a basis of $\sl_2\ltimes \R^2$,
$v_1,v_2,v_3,w_1,w_2,w_3$ a basis of $\m=V\oplus V^*$ and
$\theta_1,\theta_2,\theta_3,\sigma_1,\sigma_2,\sigma_3$ the dual basis of $\m^*=V^*\oplus V$.
In these bases the representation $\rho\in\h^*\ot\m^*\ot\m$ has the form:
\com{
 \begin{gather*}
\rho(e)=E_{12}-E_{54},\qquad \rho(f)=E_{21}-E_{45},\\
\rho(h)=E_{11}-E_{22}-E_{44}+E_{55}
\rho(x_1)=E_{13}-E_{64},\qquad \rho(x_2)=E_{23}-E_{65}.
 \end{gather*}
}
 \begin{gather*}
\rho(e)=\theta_2\ot v_1-\sigma_1\ot w_2,\qquad \rho(f)=\theta_1\ot v_2-\sigma_2\ot w_1,\\
\rho(h)=\theta_1\ot v_1-\theta_2\ot v_2-\sigma_1\ot w_1+\sigma_2\ot w_2,\\
\rho(x_1)=\theta_3\ot v_1-\sigma_1\ot w_3,\qquad \rho(x_2)=\theta_3\ot v_2-\sigma_2\ot w_3,
 \end{gather*}
\com{
 \begin{align*}
&\begin{bmatrix}0,1,0,0,0,0\\0,0,0,0,0,0\\0,0,0,0,0,0\\0,0,0,0,0,0\\0,0,0,-1,0,0\\0,0,0,0,0,0\end{bmatrix},
\begin{bmatrix}0,0,0,0,0,0\\1,0,0,0,0,0\\0,0,0,0,0,0\\0,0,0,0,-1,0\\0,0,0,0,0,0\\0,0,0,0,0,0\end{bmatrix},
\begin{bmatrix}1,0,0,0,0,0\\0,-1,0,0,0,0\\0,0,0,0,0,0\\0,0,0,-1,0,0\\0,0,0,0,1,0\\0,0,0,0,0,0\end{bmatrix}\\
&\begin{bmatrix}0,0,1,0,0,0\\0,0,0,0,0,0\\0,0,0,0,0,0\\0,0,0,0,0,0\\0,0,0,0,0,0\\0,0,0,-1,0,0\end{bmatrix},
\begin{bmatrix}0,0,0,0,0,0\\0,0,1,0,0,0\\0,0,0,0,0,0\\0,0,0,0,0,0\\0,0,0,0,0,0\\0,0,0,0,-1,0\end{bmatrix}
 \end{align*}
}
and the cocycle $\vp\in\h^*\ot\m^*\ot\h$ is (note $\vp(e)=\vp(f)=\vp(h)=0$):
 \begin{align*}
& \vp(x_1)=\tfrac23\,\sigma_2\ot e+\tfrac13\,\sigma_1\ot h+\sigma_3\ot x_1,\\
& \vp(x_2)=\tfrac23\,\sigma_1\ot f-\tfrac13\,\sigma_2\ot h +\sigma_3\ot x_2.
 \end{align*}
This gives the full action of $\h$ on the module $\g$.

Since no parameters appear in $\vp$, both the equations $\delta \vp=d\theta_m$ and equation (2) from Theorem \ref{newcohomstatement}
are linear inhomogeneous. Solving these gives the following set of brackets on $\m$, parametrized by $a_1,\ldots,a_7\in\R$:
\begin{align*}
&[v_1, v_3] = a_1 x_1+a_2 v_1, [v_1, w_1] =  -\tfrac{2}{3} v_3-a_7 w_3+a_7 h, [v_1, w_2] = 2 a_7 e,\\
&[v_1, w_3] =  \tfrac{1}{3} v_1+2 a_7 x_1, [v_2, v_3] = a_1 x_2+a_2 v_2, [v_2, w_1] = 2 a_7 f,\\
&[v_2, w_2] =  -\tfrac{2}{3} v_3-a_7 w_3-a_7 h, [v_2, w_3] =  \tfrac{1}{3} v_2+2 a_7 x_2,\\
&[v_3, w_1] = -a_3 x_2-a_4 v_2+3 a_7 w_1, [v_3, w_2] = a_3 x_1+a_4 v_1+3 a_7 w_2,\\
&[v_3, w_3] =  -\tfrac{2}{3} v_3+2 a_7 w_3, [w_1, w_2] = a_5 v_3+a_6 w_3,\\
&[w_1, w_3] = -a_5 v_2-a_6 x_2-w_1, [w_2, w_3] = a_5 v_1+a_6 x_1-w_2.
\end{align*}
Note that it is possible to see that the curvature of the space $\langle v_1,v_2,v_3 \rangle$ is degenerate already here, before solving any non-linear equations, because $[v_1,v_2]=0$. The Jacobi identities between three elements of $\m$ yield a polynomial ideal, for which a Gröbner basis is given by
\begin{align*}
&-4a_7+3a_2=0, a_7^2+3a_1=0, 3a_5a_7+2a_4+a_6=0, a_1a_4+2a_1a_6+a_3a_7=0,\\
&a_4a_7+2a_6a_7-3a_3=0, 6a_1a_5-a_4a_7-a_3=0, 9a_3a_5+2a_4^2+5a_4a_6+2a_6^2=0.
\end{align*}
There is a unique family of solutions, given by
$$
a_1 = 3 a_7^2, a_2 = 4 a_7, a_3 =-\tfrac{3}{10} a_6 a_7^2+\tfrac{3}{4} a_5 a_7, a_4 = -\tfrac{3}{5} a_7 a_6-\tfrac{1}{2} a_5.
$$
Once these are substituted into the brackets, $\g$ is a Lie algebra. A Levi decomposition of $\g$ is then given by
\begin{align*}
&\g_{ss}=\sl_3=\langle e, f, h, x_1, x_2, a_5 v_1-\tfrac43w_2, a_5 v_2+\tfrac43w_1, w_3\rangle,\\
&\g_{rad}=V=\langle a_7 x_1-\tfrac13v_1, a_7 x_2-\tfrac13v_2, v_3-3 a_7 w_3 \rangle,
\end{align*}
Now, $\h$ is embedded in $\g_{ss}$, and all embeddings of $\h$ into $\sl_3$ are equivalent up to an outer automorphism. Moreover one may verify that the subspace $\langle v_1,v_2,v_3\rangle$ is involutive modulo $\h$. Hence we obtain the result.
 \end{proof}

Next, consider the subalgebras $\mathfrak{s}^2\ltimes \R^2$ with a parameter $l$ as in Proposition~\ref{magicnumbers}. %Suppose first that $l\not=\frac{3}{2}$.

 \begin{prop}
Suppose $\h=\mathfrak{s}^2\ltimes \R^2$ and let $l\not=\frac{3}{2}$. Then we may assume $\g$ is decomposable, and we have $\mathfrak{B}(\h,\g)=(\Lambda^2 \m^\ast \otimes \m)^{\sl_3}$, i.e. $\h$ has no additional equivariant brackets, unless $l\in\{0,\frac{-3}{10},\frac{-1}{2},\frac{-3}{4},\frac{-3}{2}  \}$. If $\g=\h\oplus\m$ is a Lie algebra, then the Nijenhuis tensor of its associated almost product structure is degenerate.
 \end{prop}

 \begin{proof}

Let $t,e,x_1,x_2$ be a basis of $\mathfrak{s}^2\ltimes \R^2$,
$v_1,v_2,v_3,w_1,w_2,w_3$ a basis of $\m=V\oplus V^*$ and
$\theta_1,\theta_2,\theta_3,\sigma_1,\sigma_2,\sigma_3$ the dual basis of $\m^*=V^*\oplus V$.
In these bases the representation $\rho\in\h^*\ot\m^*\ot\m$ has the form:
 \com{
\begin{align*}
&\begin{bmatrix}\frac{-1}{2}+\frac{1}{3} l & 0 & 0 & 0 & 0 & 0 \\ 0 & \frac{1}{2}+\frac{1}{3} l & 0 & 0 & 0 & 0 \\ 0 & 0 & \frac{-2}{3} l & 0 & 0 & 0 \\ 0 & 0 & 0 & \frac{1}{2}-\frac{1}{3} l & 0 & 0 \\ 0 & 0 & 0 & 0 & \frac{-1}{2}-\frac{1}{3} l & 0 \\ 0 & 0 & 0 & 0 & 0 & \frac{2}{3} l\end{bmatrix},
\begin{bmatrix}0 & 0 & 0 & 0 & 0 & 0 \\ 1 & 0 & 0 & 0 & 0 & 0 \\ 0 & 0 & 0 & 0 & 0 & 0 \\ 0 & 0 & 0 & 0 & -1 & 0 \\ 0 & 0 & 0 & 0 & 0 & 0 \\ 0 & 0 & 0 & 0 & 0 & 0\end{bmatrix},\\
&\begin{bmatrix}0 & 0 & 1 & 0 & 0 & 0 \\ 0 & 0 & 0 & 0 & 0 & 0 \\ 0 & 0 & 0 & 0 & 0 & 0 \\ 0 & 0 & 0 & 0 & 0 & 0 \\ 0 & 0 & 0 & 0 & 0 & 0 \\ 0 & 0 & 0 & -1 & 0 & 0\end{bmatrix},
\begin{bmatrix}0 & 0 & 0 & 0 & 0 & 0 \\ 0 & 0 & 1 & 0 & 0 & 0 \\ 0 & 0 & 0 & 0 & 0 & 0 \\ 0 & 0 & 0 & 0 & 0 & 0 \\ 0 & 0 & 0 & 0 & 0 & 0 \\ 0 & 0 & 0 & 0 & -1 & 0\end{bmatrix},
\end{align*}
 }
 \com{
 \begin{gather*}
\rho(t)=(\tfrac13l-\tfrac12)E_{11}+(\tfrac13l+\tfrac12)E_{22}-\tfrac23lE_{33}
-(\tfrac13l-\tfrac12)E_{44}-(\tfrac13l+\tfrac12)E_{55}+\tfrac23lE_{66},\\
\rho(e)=E_{21}-E_{45},
\rho(x_1)=E_{13}-E_{64},
\rho(x_2)=E_{23}-E_{65},
 \end{gather*}
 }
 \begin{align*}
% \rho(t)=(\tfrac{l}3-\tfrac12)(\theta_1\ot v_1-\sigma_1\ot w_1)+(\tfrac{l}3+\tfrac12)(\theta_2\ot v_2-\sigma_2\ot w_2)
% -\tfrac{2l}3(\theta_3\ot v_3-\sigma_3\ot w_3),\\
& \rho(t)=(\tfrac{l}3-\tfrac12)\kappa_1+(\tfrac{l}3+\tfrac12)\kappa_2-\tfrac{2l}3\kappa_3,\
& \rho(e)=\theta_1\ot v_2-\sigma_2\ot w_1,\\
& \rho(x_1)=\theta_3\ot v_1-\sigma_1\ot w_3,\
& \rho(x_2)=\theta_3\ot v_2-\sigma_2\ot w_3,
 \end{align*}
where $\kappa_i=\theta_i\ot v_i-\sigma_i\ot w_i$.

By Proposition \ref{nontrivial4dcohom},  $\g=\h\oplus\m$ splits as an $\h$-module. We induce the action on the space $\La^2\m^\ast\otimes \m$. The action of the subalgebra $\h^{\{ 1\}}=\langle e, x_1, x_2 \rangle$ does not depend on $l$, and invariance with respect to this yields a subspace $(\La^2\m^\ast\otimes \m)^{\h^{\{ 1\}}}$ of dimension 16. Applying $t$ to a basis of this subspace gives a system of 32 equations, where each equation can be factored into a product of linear equations. Those factors that depend on $l$ are $10l+3,2l+1,4l+3,2l+3$ or $l$. If $l$ does not solve any of these, then $(\La^2\m^\ast\otimes \m)^\h=(\La^2\m^\ast\otimes \m)^{\sl_3}$, but if $l$ does solve some factor, then $(\La^2\m^\ast\otimes \m)^\h$ properly contains $(\La^2\m^\ast\otimes \m)^{\sl_3}$.

The space $(\La^2\m^\ast\otimes \h)^\h$ is treated similarly, and leads to the same values of $l$.

Next we will show that the Nijenhuis tensor is degenerate. First consider the case $l=0$.
In this case, $\mathfrak{s}^2=\mathfrak{b}^2$ is the Borel subalgebra.
The space of equivariant brackets has dimension 9, of which 2 are vertical and 7 are horizontal.
The most general equivariant brackets are then:
 \begin{align*}
& [v_1, v_2] = a_1  w_3, [v_1, v_3] = a_8  x_1+a_3  v_1-a_1  w_2, [v_2, v_3] = a_8  x_2+a_3  v_2+a_1  w_1, \\
& [v_1, w_1] = (a_7-a_9)  w_3, [v_2, w_2] = (a_7-a_9)  w_3, [v_3, w_3] = a_7  w_3, \\
& [v_3, w_1] = -a_4  x_2-a_5  v_2+a_9  w_1, [v_3, w_2] = a_4  x_1+a_5  v_1+a_9  w_2, \\
& [w_1, w_2] = a_6  w_3+a_2  v_3, [w_1, w_3] = -a_2  v_2, [w_2, w_3] = a_2  v_1.
 \end{align*}
For $l=\frac{-3}{10}$ the space of equivariant brackets has dimension 4, of which 1 is vertical and 3 are horizontal, and the most general equivariant brackets are:
 \begin{align*}
& [v_1, v_2] = a_1  w_3, [v_1, v_3] = -a_1  w_2, [v_2, v_3] = a_1  w_1, [v_3, w_2] = a_3  w_3, \\
& [w_1, w_2] = a_4  x_2+a_2  v_3,  [w_1, w_3] = -a_2  v_2, [w_2, w_3] = a_2  v_1.
 \end{align*}
For $l=\frac{-3}{4}$ the space of equivariant brackets has dimension 4, of which 1 is vertical and 3 are horizontal, and the most general equivariant brackets are:
 \begin{align*}
& [v_1, v_2] = a_1  w_3, [v_1, v_3] = a_3  x_2-a_1  w_2, [v_2, v_3] = a_1  w_1, [v_3, w_2] = a_4  v_2,\\
& [w_1, w_2] = a_2  v_3, [w_1, w_3] = -a_2  v_2, [w_2, w_3] = a_2  v_1.
 \end{align*}
For $l=\frac{-3}{2}$ the space of equivariant brackets has dimension 9, of which 2 are vertical and 7 are horizontal, and the most general equivariant brackets are:
 {\small\begin{align*}
& [v_1, v_2] = a_1  w_3, [v_1, v_3] = a_6  v_2-a_1  w_2, [v_1, w_1] = a_7  v_2,[v_1, w_2] = -a_9  x_2-a_3  w_3+a_8  v_1,\\
& [v_2, v_3] = a_1  w_1,  [v_2, w_2] = (a_7+a_8)  v_2, [v_3, w_2] = a_9  e+a_3  w_1+a_8  v_3, [v_3, w_3] = a_7  v_2,\\
& [w_1, w_2] = -a_4  e-a_5  w_1+a_2  v_3, [w_1, w_3] = -a_2  v_2, [w_2, w_3] = a_4  x_2+a_5  w_3+a_2  v_1.
 \end{align*}}\!\!
For $l=\frac{-1}{2}$ the space of equivariant brackets has dimension 7, of which 3 are vertical and 4 are horizontal, and the most general equivariant brackets are:
 \begin{align*}
& [v_1, v_2] = a_1  w_3, [v_1, v_3] = a_4  w_3-a_1  w_2, [v_2, v_3] = a_1  w_1,  [v_1, w_1] = (a_5-a_7)  x_2, \\
& [v_1, w_2] = a_7  x_1, [v_2, w_2] = a_5  x_2,  [v_3, w_1] = a_7  e, [v_3, w_2] = a_6  x_2+a_7  t, \\
& [v_3, w_3] = a_5  x_2, [w_1, w_2] = a_3  v_2+a_2  v_3, [w_1, w_3] = -a_2  v_2, [w_2, w_3] = a_2  v_1.
 \end{align*}
For all these cases, we get that the curvature of $\langle v_1,v_2,v_3\rangle$ is degenerate if $a_1=0$, and the curvature of $\langle w_1,w_2,w_3\rangle$ is degenerate if $a_2=0$. However, we have
$$
\text{Jac}(v_1,v_2,w_1)=a_1 a_2 w_3,
$$
which yields the equation $a_1 a_2=0$. Therefore at least one distribution has degenerate curvature, and
the Nijenhuis tensor degenerates.
 \end{proof}

Finally we treat the last exceptional parameter $l$.

 \begin{prop}
Suppose $\h=\mathfrak{s}^2\ltimes \R^2$, $l=\frac{3}{2}$. If $\g$ is a Lie algebra with isotropy $\h$,
then the Nijenhuis tensor of its associated almost product structure is degenerate.
 \end{prop}

 \begin{proof}
Let $t,e,x_1,x_2$ be the usual basis of $\h$, and $v_1,v_2,v_3,w_1,w_2,w_3$ a basis of $\g$ as in the previous proof.
By Proposition \ref{magicnumbers}, $\dim \text{H}^1(\h,\Hom(\m,\h))=6$.
The following representation matrices of $\h$ on $\g$ realize a general element of this cohomology
in the basis $t,e,x_1,x_2,v_1,v_2,v_3,w_1,w_2,w_3$:

 $$
\rho(t)=\op{diag}(0,1,1,2,0,1,-1,0,-1,1)
 $$

\[
\rho(e)=
{\footnotesize \left[
\begin{array}{cccc|cccccc}
0 & 0 & 0 & 0 & 0 & 0 & -3 c_4 & 0 & 11 c_1 & 0 \\
-1 & 0 & 0 & 0 & 14 c_2 & 0 & 0 & 29 c_1 & 0 & 0 \\
0 & 0 & 0 & 0 & 4 c_4 & 0 & 0 & c_3 & 0 & 0 \\
0 & 0 & 1 & 0 & 0 & 5 c_4 & 0 & 0 & 0 & 17 c_1 \vphantom{\frac{2^2}{2^2}}\\  \hline
0 & 0 & 0 & 0 & 0 & 0 & 0 & 0 & 0 & 0 \vphantom{\frac{2^2}{2^2}}\\
0 & 0 & 0 & 0 & 1 & 0 & 0 & 0 & 0 & 0 \\
0 & 0 & 0 & 0 & 0 & 0 & 0 & 0 & 0 & 0 \\
0 & 0 & 0 & 0 & 0 & 0 & 0 & 0 & -1 & 0 \\
0 & 0 & 0 & 0 & 0 & 0 & 0 & 0 & 0 & 0 \\
0 & 0 & 0 & 0 & 0 & 0 & 0 & 0 & 0 & 0
\end{array}
\right]},
 \]
 \[
\rho(x_1)=
{\footnotesize \left[
\begin{array}{cccc|cccccc}
0 & 0 & 0 & 0 & 0 & 0 & 11 c_2 & 0 & 3 c_5 & 0 \\
0 & 0 & 0 & 0 & c_6 & 0 & 0 & 4 c_5 & 0 & 0 \\
-1 & 0 & 0 & 0 & -29 c_2 & 0 & 0 & -14 c_1 & 0 & 0 \\
0 & -1 & 0 & 0 & 0 & -17 c_2 & 0 & 0 & 0 & 5 c_5 \vphantom{\frac{2^2}{2^2}}\\  \hline
0 & 0 & 0 & 0 & 0 & 0 & 1 & 0 & 0 & 0 \vphantom{\frac{2^2}{2^2}}\\
0 & 0 & 0 & 0 & 0 & 0 & 0 & 0 & 0 & 0 \\
0 & 0 & 0 & 0 & 0 & 0 & 0 & 0 & 0 & 0 \\
0 & 0 & 0 & 0 & 0 & 0 & 0 & 0 & 0 & 0 \\
0 & 0 & 0 & 0 & 0 & 0 & 0 & 0 & 0 & 0 \\
0 & 0 & 0 & 0 & 0 & 0 & 0 & -1 & 0 & 0
\end{array}
\right]},
 \]
 \[
\rho(x_2)=
{\footnotesize \left[
\begin{array}{cccc|cccccc}
0 & 0 & 0 & 0 & 0 & 0 & 0 & 0 & 0 & 0 \\
0 & 0 & 0 & 0 & 0 & 0 & 3 c_2 & 0 & c_5 & 0 \\
0 & 0 & 0 & 0 & 0 & 0 & c_4 & 0 & -3 c_1 & 0 \\
-2 & 0 & 0 & 0 & 2 c_2 & 0 & 0 & -2 c_1 & 0 & 0 \vphantom{\frac{2^2}{2^2}}\\  \hline
0 & 0 & 0 & 0 & 0 & 0 & 0 & 0 & 0 & 0 \vphantom{\frac{2^2}{2^2}}\\
0 & 0 & 0 & 0 & 0 & 0 & 1 & 0 & 0 & 0 \\
0 & 0 & 0 & 0 & 0 & 0 & 0 & 0 & 0 & 0 \\
0 & 0 & 0 & 0 & 0 & 0 & 0 & 0 & 0 & 0 \\
0 & 0 & 0 & 0 & 0 & 0 & 0 & 0 & 0 & 0 \\
0 & 0 & 0 & 0 & 0 & 0 & 0 & 0 & -1 & 0
\end{array}
\right]}.
 \]

Due to the linear equation $\delta \vp = d \theta_\m$, we immediately get $c_3=c_4=c_5=c_6=0$, as only a 2D subspace of the cohomology satisfies $[\delta \vp]=0$. Equation (2) from Theorem \ref{newcohomstatement} is then quadratic in $c_1$ and $c_2$, but a Gröbner basis computation implies $c_1 c_2 =0$. Hence we may normalize to $c_1=0,c_2=1$ or $c_1=1,c_2=0$.

After this normalization, equation (2) from Theorem \ref{newcohomstatement} is linear, and is easily solved. In both cases the resulting brackets depend on one parameter $\a$, and satisfy the Jacobi identities on $\m$ without any further constraints.

In the case $c_1=0,c_2=1$, the brackets are the following:
 {\small\begin{align*}
& [v_1, v_2] =  \a w_3-51 e-28 v_2, [v_1, v_3] = - \a w_2-29 v_3, [v_2, v_3] =  \a w_1, [v_1, w_1] = -20 w_1, \\
& [v_1, w_2] = 11 w_2,[v_1, w_3] = 9 w_3, [v_2, w_2] = -17 w_1, [v_3, w_3] = -20 w_1, \\
& [t, v_2] = v_2, [t, v_3] = -v_3, [t, w_2] = -w_2, [t, w_3] = w_3, [e, v_1] = v_2+14 e, [e, w_2] = -w_1,\\
& [x_1, v_1] = -29 x_1, [x_1, v_2] = -17 x_2, [x_1, v_3] = v_1+11 t, [x_1, w_1] = -w_3, [x_2, v_1] = 2 x_2,\\
& [x_2, v_3] = v_2+3 e, [x_2, w_2] = -w_3, [t, e] = e, [t, x_1] = x_1, [t, x_2] = 2 x_2, [e, x_1] = x_2.
 \end{align*}}\!\!
Note that $\langle w_1,w_2,w_3 \rangle$ is an Abelian subalgebra, hence it has vanishing curvature. In the case $c_1=1,c_2=0$, the brackets are the following:
 {\small\begin{align*}
& [v_1, w_1] = -20 v_1, [v_2, w_1] = 9 v_2, [v_2, w_2] = -20 v_1, [v_3, w_1] = 11 v_3, [v_3, w_3] = -17 v_1, \\
& [w_1, w_2] =  \a v_3+29 w_2, [w_1, w_3] = - \a v_2+51 x_1+28 w_3, [w_2, w_3] =  \a v_1, \\
& [t, v_2] = v_2, [t, v_3] = -v_3, [t, w_2] = -w_2, [t, w_3] = w_3, [e, v_1] = v_2, [e, w_2] = -w_1+11 t, \\
& [e, w_1] = 29 e, [e, w_3] = 17 x_2, [x_1, v_3] = v_1, [x_1, w_1] = -w_3-14 x_1, [x_2, v_3] = v_2, \\
& [x_2, w_1] = -2 x_2, [x_2, w_2] = -w_3-3 x_1,  [t, e] = e, [t, x_1] = x_1, [t, x_2] = 2 x_2, [e, x_1] = x_2.
 \end{align*}}\!\!
In this case, $\langle v_1,v_2,v_3 \rangle$ is an Abelian subalgebra, and has vanishing curvature. In both solutions, one of the distributions has vanishing curvature, and therefore the associated Nijenhuis tensor is degenerate.
 \end{proof}

\subsection{The Submaximal model is globally transitive}\label{S54}
In the previous subsection, we proved that the homogeneous space $Sp(4,\R)/(SL_2\times\R)$
with unique almost product structure is the only locally homogeneous non-degenerate para-complex manifold $(M^6,J)$
with the symmetry algebra of dimension $d\in[10,14)$.
The isotropy algebra $\gl_2$ is embedded into $\sp(4,\R)$ via the block-diagonal embedding
$\gl_2=\sl_2\oplus\R=\sp(2,\R)\oplus\so(1,1)\subset\sp(2,\R)\oplus\sp(2,\R)$.

Our next goal now is to prove that no intransitive examples with symmetry dimension $10\leq d<14$ exist,
which means that the type $\sp(4,\R)/\gl_2$ gives a submaximal model and that all submaximal structures are locally transitive.
Then we show that the submaximal model admits no singular orbits, which means that the complete global submaximal models are homogeneous spaces.

When the symmetry group $G$ is not locally transitive, the $G$-manifold $M$ (or its invariant open subset) is not (naturally, locally) homogeneous. Therefore the full range of algebraic tools we used in the previous section is unavailable to us. Instead, we can find a foliation by $G$-orbits in a neighbourhood of any regular point $x\in M$. Each leaf is a local homogeneous space of $G=\op{Aut}(J)$ in its own right. We may therefore investigate the existence of lower dimension homogeneous spaces $O$ whose isotropy algebra admits the existence of an invariant non-degenerate Nijenhuis tensor on the tangent space $\m$ of a regular point of $M$. This means that the full isotropy representation $\m$ must be one of those discussed in the previous section.

 \begin{prop}\label{locallytransitive}
Let $\g=\sym(J)$ be the symmetry algebra of a non-degenerate para-complex structure $J$ with $\dim \g \ge 10$. Then $\g$ is locally transitive.
 \end{prop}

 \begin{proof}
Let us exclude the case $\h=\sl_3$ considered in Section \ref{S4}.
Suppose $\g$ is not locally transitive. The tangent space $T_oO=\tao$ of the orbit through $o\in M$ must be an invariant subspace of $\m$ for the isotropy algebra $\h$. The isotropy $\h$ is still represented effectively (now on $\tao$) as before, so the dimension of the symmetry algebra $\g$ is $\dim\g=\dim\tao+\dim\h$.
This means the possible pairs $(\h,\tao)$ that have combined dimension $\dim\g\ge 10$ are the following:
 \begin{itemize}
\item $ \h= \p_1$, $ \dim\tao=5\vee 4$.
\item $ \h= \p_{12}$, $ \dim\tao=5$.
\item $ \h= \sl_2\ltimes \R^2$, $ \dim\tao=5$.
 \end{itemize}
The representation of $\h$ is via restriction of $\sl_3$-representation $\m=V\oplus V^*$ and it is easy to see that
in the first two cases $\m$ has no invariant subspace of dimension 4 or 5.
In the last case there exists a unique 5-dimensional invariant subspace $\mathfrak{o}=U\oplus\Delta_-$,
where $U\subset V=\Delta_+$ is a plane, and we identify $V^{(*)}\simeq\Delta_\pm$.
However since $\Delta_-\subset\tao$, due to non-degenerate curvature of the
distribution, at the regular domain in $M$ where the orbits foliate the space we get $TM=[\Delta_-,\Delta_-]\subset\tao$,
a contradiction.
 \end{proof}

 \begin{prop}
Let $\g=\sym(J)$ be the symmetry algebra of a non-degenerate para-complex structure $J$, $\dim \g \ge 10$.
Then $\g$ has no singular orbits. % including invariant points.
 \end{prop}

 \begin{proof}
We can assume $\dim\g^*<14$ as the maximally symmetric case is already resolved.
Then the previous section and Proposition \ref{locallytransitive} imply that $\g=\sp(4,\R)$.

By Theorem \ref{finitedim}, the isotropy representation is faithful, and
by Corollary \ref{dim14} its image is a proper subalgebra $\h\subset\sl_3$,
unless the orbit $O=o$ is a fixed point. Such points are excluded by an application of 
the Thurston stability theorem \cite{T} as in Section \ref{S42}: indeed,
$\mathfrak{sp}(4,\mathbb{R})$ does not have 3-dimensional faithful representations $V$ or $V^*$. 
Hence, the isotropy algebra $\h$ is also a proper subalgebra of $\sp(4,\R)$. 
The proper subalgebras of $\g$ have dimension at most 7, but $\sl_3$ has no such proper subalgebras.

Therefore $4\leq\dim O\leq 5$ and $\dim\h=10-\dim O\in[5,6]$.

Consider the case $\dim O =4$, $\dim \h=6$.
The only subalgebras of $\sl_3$ of dimension 6 are $\p_1$ and $\p_2$, both of which are isomorphic to $\gl_2\ltimes \R^2$. But $\sp(4,\R)$ has no subalgebra isomorphic to this, so this case is not realizable.

Now suppose $\dim O =5$. The subalgebras of $\sl_3$ of dimension 5 are $\p_{12}$ and $\sl_2\ltimes \R^2$. The latter cannot occur, because there is no embedding of $\sl_2$ into $\sp(4,\R)$ which normalizes a $2$-dimensional abelian subalgebra (there are embeddings that stabilize $2D$ submodules of the correct type, but these submodules generate Heisenberg algebras). Thus we must have $\h=\p_{12}\simeq \R^2\ltimes \mathfrak{heis}(3)$. This can be uniquely embedded, such that it is contained in $\p_1\subset \sp(4,\R)$, where $\p_1 \simeq \gl_2\ltimes \mathfrak{heis}(3)$, by mapping $\R^2$ to the diagonal subalgebra of $\gl_2$.

Let $\tao=T_x O\subset \m$ be the invariant tangent space of the orbit.
The decomposition $\m=V\oplus V^\ast$ is invariant. Let $\tao_+ = \tao\cap V$ and $\tao_- = \tao \cap V^\ast$.
These are invariant subspaces of dimension at least 2. Thus the isotropy representation $\tao$ must admit either
an invariant $2$-plane with an invariant complement (3-plane), or two distinct invariant $2$-planes.
The isotropy representation $\tao=\sp(4,\R)/\R^2\ltimes \mathfrak{heis}(3)$ admits one invariant 2-plane,
but neither of the possible complements. This is the final contradiction, and we are done.
%Then the intersections $\Pi_+, \Pi_-$ have dimension at least 2, and $\dim \h =5$. If either $\Pi_+$ or $\Pi_-$ has dimension 3, then it will generate $\m$ under Lie bracket due to the non-degenerate curvature of $V$ and $V^\ast$, which is not possible since the orbit has positive codimension. Therefore $\dim\Pi_+=\dim\Pi_-=2$, and since these are invariant 2-planes, $\h=\p_{12}\subset\sl_3$. But $\p_{12}$ does not admit any 5-dimensional invariant subspace $\tao$.
\end{proof}
This concludes the proof of Theorem \ref{Thm2}.

\section{Topology and Geometry of the global models}\label{S6}
In this section we discuss the maximal and submaximal symmetry models from Theorems \ref{Thm1} 
and \ref{Thm2} in more detail, expanding on Remark \ref{Rk1}, and also describe the invariant
paracomplex structures on $\mathbb{S}^3\times\mathbb{S}^3$.

\subsection{Revising the models}
Let us visualize the para-complex structure $J$ via the root diagram of the maximal and submaximal
Lie algebras of symmetries. Below are the root systems of $\g_2^*$ and $\mathfrak{sp}(4,\R)$ respectively.

%\vspace{1cm}

 \begin{center}
\includegraphics[height=5.0cm]{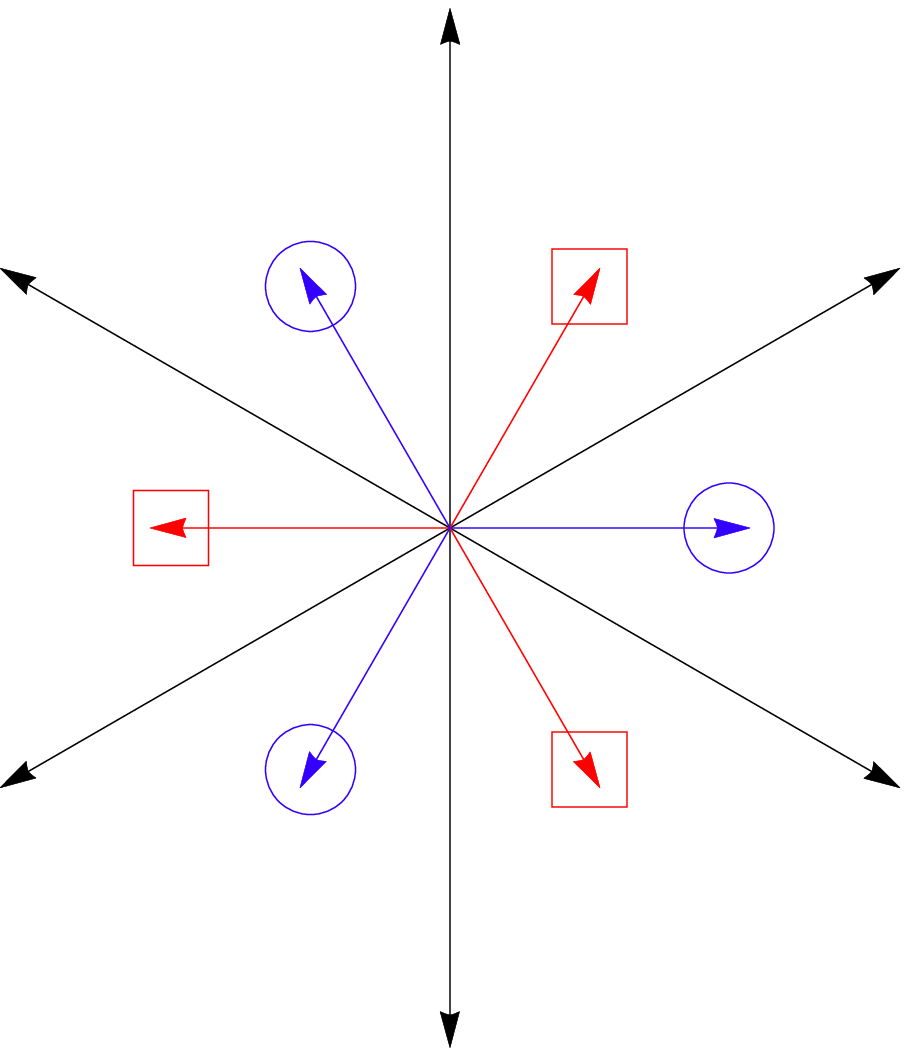}\ \hskip1cm
\includegraphics[height=5.0cm]{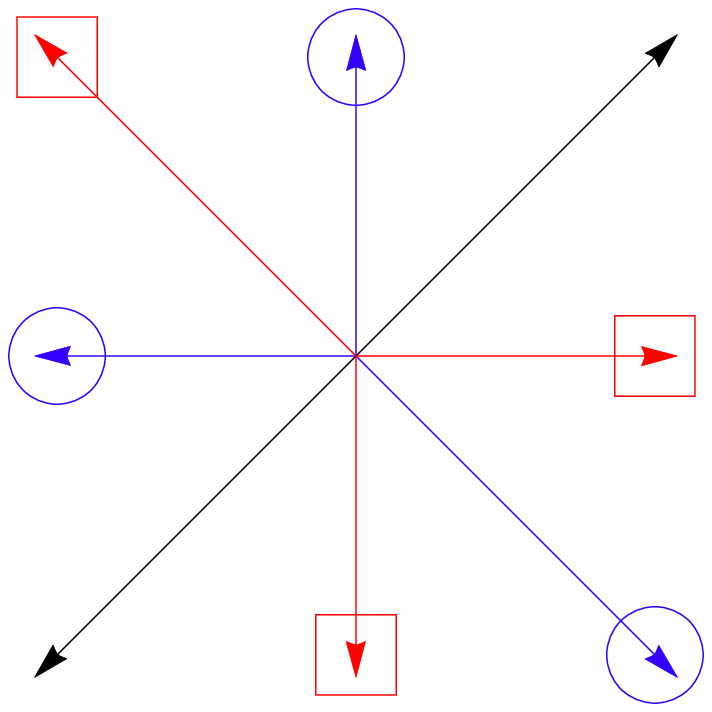}
 \end{center}

%\vspace{1cm}

The black (unmarked) arrows together with the Cartan subalgebra form the root system of the
subalgebra $\h$. It is $\mathfrak{sl}_3$ and $\mathfrak{gl}_2$ respectively. The root vectors
corresponding to coloured (marked) roots span the representation space $\m$
(one easily confirms that $\m$ is a module over $\h$).
The red vectors (marked with squares at the endpoints) correspond to $V$
and the blue ones (marked with circles at the endpoints) correspond to $V^*$
(up to interchange $V\leftrightarrow V^*$).

Since the pairwise sums of the red vectors coincide with the blue ones and otherwise around, we conclude
that the brackets $\La^2V\to V^*$ and $\La^2V^*\to V$ are isomorphisms. Thus the linear operator $J:\m\to\m$
for which $V$ and $V^*$ are $\pm1$ eigenspaces defines a non-degenerate para-complex structure.

\subsection{Topology description}
For the maximal symmetry model let us explain a diffeomorphism of $G_2^*/SL_3$ to
$\mathbb{S}^{3,3}=\mathbb{S}^3\times\R^3$:
$G_2^*$ is represented as a subgroup of $SO(3,4)$ acting on $\R^{3,4}$.
The stabilizer of a point $x$ with $\|x\|^2=\sum_4^7x_i^2-\sum_1^3x_i^2=1$ is $SL_3$
(while that of a point $x$ with $\|x\|^2=-1$ is $SU(1,2)$ \cite{Ka}).
This unit sphere $x_4^2+x_5^2+x_6^2+x_7^2=x_1^2+x_2^2+x_3^2+1$ has the required topological type.

To approach the submaximally symmetric model, let us begin with the linear group quotient $Sp(4,\R)/(SL_2\times\R)$.
 \com{
The action of the group $Sp(4,\R)$ on the Lagrangian 2-planes in the standard symplectic $\R^4$
is transitive, and we get $Sp(4,\R)/P_2=\Lambda(2)$, where $P_2$ is the parabolic subgroup
stabilizing a Lagrangian plane. This subgroup is 1-graded: $P_2=GL_2\ltimes V^3$, where $V^3$ is
an Abelian subalgebra (and an irreducible $SL_2$ representation). Thus the exact sequence
 $$
e\to P_2/GL_2\longrightarrow Sp(4,\R)/GL_2\longrightarrow Sp(4,\R)/P_2\to e
 $$
implies that $Sp(4,\R)/GL_2=\Lambda(2)\tilde{\times}V^3$.

It is well known that the oriented Lagrangian $\Lambda_+(2)=U(2)/SO(2)\simeq\mathbb{S}^1\times\mathbb{S}^2$
and it doubly covers the Lagrangian Grassmanian $\Lambda(2)=U(2)/O(2)\simeq\mathbb{S}^1\tilde{\times}\mathbb{S}^2$.
The latter is the quotient $\mathbb{S}^1\times\mathbb{S}^2/\Z_2$, where $\Z_2$ acts diagonally,
and is represented as either a nontrivial $\mathbb{S}^1$-bundle over $\R P^2$ or as
a nontrivial $\mathbb{S}^2$-bundle over $\mathbb{S}^1$.
 }
By the results on the isotropy representation from \S\ref{S54}, the subgroup is embedded into $Sp(4,\R)$ as follows.
Consider a splitting $\R^4=\R^2\oplus\R^2$ into a sum of two symplectic planes. Let $SL_2=Sp(2,\R)$ act on the first
summand and $\R=SO_+(1,1)\subset Sp(2,\R)$ act on the second summand (this explains why the action is not by
$GL_2$ but by $SL_2\times\R$).

The subgroup $SL_2\times\R$ is the reductive part (corresponding to the zero grading) of the parabolic
subgroup $P_1\subset Sp(4,\R)$ stabilizing a ray in $\R^4$. Clearly the space of rays is $\mathbb{S}^3=Sp(4,\R)/P_1$.
The subgroup $P_1$ determines a $|2|$-grading on the group, in particular the nilradical of $P_1$ 
is the Heisenberg 3-group. This implies that $P_1/(SL_2\times\R)$ is topologically a vector 3-space 
$V^3$, but it has an invariant splitting $V^3=V^1\oplus V^2$. Thus the exact sequence
 $$
e\to P_1/(SL_2\times\R)\longrightarrow Sp(4,\R)/(SL_2\times\R)\longrightarrow Sp(4,\R)/P_1\to e
 $$
implies that $Sp(4,\R)/(SL_2\times\R)=\mathbb{S}^3\tilde{\times}V^3$ is a non-trivial bundle. Due to what was said above,
this bundle splits into a 1-dimensional trivial ($V^1$) and a 2-dimensional non-trivial ($V^2$) subbundles.

Let us pass to a description of all complete submaximal models, which are those homogeneous spaces with the same isotropy
data as for $(\mathfrak{sp}(4,\R),\mathfrak{gl}_2)$ described above. First notice that the subalgebra
$\mathfrak{gl}_2=\mathfrak{sl}_2\oplus\R\subset\mathfrak{sp}(4,\R)$ is self-normalizing. The normalizer of
the subgroup $H=SL_2\times\R\subset Sp(4,\R)$ is obtained by adding the element $\op{diag}(1,1,-1,-1)$ that
coincides with $-\1=\op{diag}(-1,-1,-1,-1)$ modulo the subgroup
(in other words, enlarging by it is equivalent to passing from $H=SL_2\times\R_+$ to $SL_2\times\R^\times$).
Thus the only centralizer element
not belonging to the subgroup is $-\1$ (modulo the subgroup center that is $\Z_2\times\R$).
Consequently, the only quotient of $Sp(4,\R)=Spin(2,3)$ carrying the required geometry is
$Spin(2,3)/\Z_2=SO_+(2,3)$. The subgroup projects isomorphically
(notice that if we take a larger subgroup $H'=SL_2\times\R^\times$, the projection has kernel $\Z_2$,
and $H'$ projects to the same $SL_2\times\R^\times/(\Z_2)_{\op{diag}}=H$);
in fact the embedding $H\hookrightarrow SO_+(2,3)$ is $(A,\lambda)\mapsto\op{diag}(\lambda A,\lambda^{-1}A,1)$,
where the diagonal square blocks have sizes $2,2,1$ and the signature of each $2\times2$ block is $(1,1)$.
Consequently we get the quotient $SO_+(2,3)/H=Sp(4,\R)/H'=\R P^3\tilde{\times}V^3$.

The fundamental group of the linear symplectic group is $\pi_1(Sp(4,\R))=\Z$ with the generator
being given through our subgroup embedding $SO(2)\subset SL_2\subset Sp(4,\R)$.
Thus when lifting the group the subgroup also lifts, and we get the quotient of non-algebraic
simply-connected groups equal to our above quotient
 $$
\widetilde{Sp}(4,\R)/(\widetilde{SL}_2\times\R)=Sp(4,\R)/(SL_2\times\R).
 $$
Now $\widetilde{Sp}(4,\R)$ has center $\mathcal{Z}=\Z\oplus\Z_2$ (generated by $\pi_1(SO(2))$
and the element $-\1$ of $Sp(4,\R)$),
and since the normalizer of $\widetilde{SL}_2\times\R$ is generated by this subgroup and $-\1$,
the only admissible quotients are by the subgroups of $\mathcal{Z}$. However any subgroup $\Gamma$ of the first
factor gives the same homogeneous space
$(\Gamma\setminus\widetilde{Sp}(4,\R))/(\Gamma\setminus\widetilde{SL}_2\times\R)=Sp(4,\R)/(SL_2\times\R)$.
Thus the only non-trivial quotient arises from the second factor ($\Z_2$), so the only
two submaximal models are $\mathbb{S}^3\tilde{\times}V^3$ and $\R P^3\tilde{\times}V^3$.

\subsection{Nearly para-K\"ahler geometry}

Every non-degenerate para-complex ma\-ni\-fold $(M,J)$ has a compatible metric of signature $(3,3)$.
Indeed, the tangent space decomposes as $TM=\Delta_-\oplus\Delta_+$ and the two summands are dual of
each other. Thus their pairing gives the required metric $g(u,v)=\langle\pi_-u,\pi_+v\rangle$,
where $\pi_\pm:TM\to\Delta_\pm$ are the projections.
This structure satisfies the relation % $J^*g=-g$, i.e.\
$g(u,v)+g(Ju,Jv)=0$.
Associated to this is the skew-symmetric form $\omega(u,v)=g(u,Jv)$.
Non-degeneracy of $J$ implies that $\omega$ is almost symplectic.
Thus non-degenerate para-complex structures induce almost para-Hermitian geometry.

Since $N_J$ is non-degenerate, the 2-form is never symplectic. In this case we can address the
issue of whether the structure $(g,J,\omega)$ on $M$ is strictly nearly para-K\"ahler, which means that
$\nabla\omega$ is totally skew-symmetric (equivalently this means $\nabla_X(J)(X)=0$).
Since $p(\nabla\omega)=\frac13d\omega$, where $p:T^*\ot\La^2T^*\to\La^3T^*$ is the projection,
this implies that $\nabla\omega=\frac13d\omega\neq0$.

Consider now the maximally symmetric model $G_2^*/SL_3$. In this case both
$\h=\mathfrak{sl}_3$ modules $\m^*\ot\La^2\m^*$ and $\La^3\m^*$ contain precisely
2 trivial summands (for $\m=V+V^*$ the decomposition is computed straightforwardly and verified in LiE).
Thus $\bar{p}:(\m^*\ot\La^2\m^*)^\h\to(\La^3\m^*)^\h$ is an isomorphism, and so
the $G_2^*$-invariant structure $(g,J,\omega)$ is strictly nearly para-K\"ahler.

Next consider one of the two models of submaximal symmetry, e.g.\ simply-connected $Sp(4,\R)/(SL_2\times\R)$.
In this case the trivial module $(\m^*\ot\La^2\m^*)^\h$ is 4-dimensional, while
dimension of $(\La^3\m^*)^\h$ is 2. Thus the corresponding projection $\bar{p}$ is not injective
and we cannot conclude strict nearly para-K\"ahler property by pure representation theoretic arguments.
This however can be observed by straightforward computations, and we exploited \textsc{Maple} to facilitate those.
We conclude the following.

 \begin{prop}
The maximal and both submaximal symmetry models are strictly nearly para-K\"ahler.
 \end{prop}

If the symmetry group of a non-degenerate para-Hermitian structure has $\dim G=9$, the homogeneous
space $G/H$ no longer needs be nearly para-K\"ahler: there are models that satisfy the corresponding
property and models that violate it.

\subsection{Compact sub-sub-maximal model}

The product of spheres $\mathbb{S}^3\times\mathbb{S}^3=SU(2)^3/SU(2)_{\op{diag}}$
has an obvious integrable paracomplex structure $J$, obtained from the double fibration over
each factor. However it also possesses a two-parameter family of invariant nondegenerate 
paracomplex structures $J$, which are constructed as follows.

Let $\h=\mathfrak{su}_2$ be diagonally embedded 
$\{(v,v,v)\}\subset\g=\mathfrak{su}_2\oplus\mathfrak{su}_2\oplus\mathfrak{su}_2$.
As $\h$-invariant complement choose $\m=\{(v_1,v_2)\}\simeq\{(v_1,v_2,0)\}\subset\g$.
All $\h$-equivariant operators $J:\m\to\m$ have the form $J(v,0)=(rv,tv)$ and then
$J(0,v)=(\frac{1-r^2}{t}v,-rv)$ is derived from the condition $J^2=\1$ ($t\neq0$).
This general form $J(v_1,v_2)=(rv_1+\frac{1-r^2}{t}v_2,tv_1-rv_2)$ 
together with its limiting cases $J(v_1,v_2)=(\pm v_1+sv_2,\mp v_2)$ give 
all invariant paracomplex structures on $\m$.

Computing the Nijenhuis tensor we observe that $J$ is integrable if $r=\pm1$ and $r-t=\pm1$ 
(so altogether four possibilites for the pair $(r,t)$, all being equivalent). 
If only one of the equalities hold, the paracomplex structure is degenerate but non-integrable.
If $r\neq\pm1$ and $r-t\neq\pm1$, then $J$ is nondegenerate. 

For each almost para-complex structure $J$, there exists exactly one (up to scale) metric $g$ on $\m$ 
that is $\h$-invariant and $J$-compatible, the latter meaning $J^*g=-g$. Then we 
obtain almost Hermitian triple $(g,J,\oo)$, where $\oo(X,Y)=g(X,JY)$, and we can study its
integrability properties. Surprisingly, for all values of the parameters $r,t$ 
 % even for those for which $J$ is degenerate, 
the triple is not nearly para-K\"ahler.
Even more, none of the arising metrics $g$ are Einstein (for any parameter value). 

To compare, in the almost complex case studied in \cite{AKW}, the corresponding 
two-parametric non-degenerate structure $\pm J$ on $\mathbb{S}^3\times\mathbb{S}^3$
is strictly nearly K\"ahler (SNK) for one value of parameters $(r,t)$,
and is pseudo-SNK even in more cases 
(when $J$ is paracomplex, the corresponding equations give only imaginary
solutions for the parameters $r,t$).

\section{Local homogeneity in the case $\dim\g=9$}
Assume that the action of the symmetry pseudogroup $G$ on $(M,J)$ or its Lie algebra $\g$ is not locally
transitive, i.e.\ there are no open orbits and so there are local invariants $I$ of the
foliations by $G$-orbits. We continue writing $\m=T_oM$ for the tangent space representation of the isotropy
algebra $\h$ ($o\in M$ is a chosen point), and we write $\mathfrak{o}\subset\m$ for the tangent to the orbit $T_o(G\cdot o)$.

In what follows we restrict to the set $U_\text{reg}\subset M$ of regular points in a neighborhood of which
the orbits fiber the space; this set $U_\text{reg}$ is open and dense in $M$.
Assuming $o\in U_\text{reg}$ and denoting by $\h$ the isotropy at $o$, the following properties hold:
 \begin{enumerate}
\item The normal bundle $\nu=\m/\mathfrak{o}$ is a trivial $\h$-representation\footnote{If $o\in M$
belongs to a singular orbit, this fact is usually false.}
(this is because the differentials of the invariants $d_oI$ span $\nu^*$, cf. \cite[Lemma 3]{KW}).
\item We have $\Delta_+\not\subset\mathfrak{o}$ and $\Delta_-\not\subset\mathfrak{o}$.
Indeed, if $\Delta_\e\subset\mathfrak{o}$,
then $\m=[\Delta_\e,\Delta_\e]\subset[\mathfrak{o},\mathfrak{o}]=\mathfrak{o}$
(the tangent distribution to the foliation by orbits is integrable).
\item We have $\mathfrak{o}_\pm=\mathfrak{o}\cap\Delta_\pm\neq0$. Indeed, if say $\mathfrak{o}_+=0$, then
$\Delta_+\hookrightarrow\nu$ is a trivial $\h$-representation, whence $\Delta_-=\Delta_+^*$ and $\m$ are also trivial.
\item The bundles $\nu_\pm=\Delta_\pm/\mathfrak{o}_\pm\subset\nu$ are trivial $\h$-representations.
 \end{enumerate}

The later claim follows from the following commutative diagram
(and its twin obtained by interchanging the signs) consisting of exact rows and columns
 \com{
$$
\begin{CD}
@. 0 @. 0 @. 0 \\
@. @VVV @VVV @VVV \\
0 @>>> \mathfrak{o}_- @>>> \mathfrak{o} @>>> \hat{\mathfrak{o}}_+ @>>> 0 \\
@. @VVV @VVV @VVV \\
0 @>>> \Delta_- @>>> \m @>>> \Delta_+ @>>> 0 \\
@. @VVV @VVV @VVV \\
0 @>>> \nu_- @>>> \nu @>>> \hat{\nu}_+ @>>> 0 \\
@. @VVV @VVV @VVV \\
@. 0 @. 0 @. 0 \\
\end{CD}
 $$
}
$$
\begin{array}{ccccccccc}
&& 0 && 0 && 0 \\
&& \downarrow && \downarrow && \downarrow \\
0 & \longrightarrow & \,\,\mathfrak{o}_- & \longrightarrow & \mathfrak{o} & \longrightarrow & \,\,\hat{\mathfrak{o}}_+ & \longrightarrow & 0 \\
&& \downarrow && \downarrow && \downarrow \\
0 & \longrightarrow & \,\,\Delta_- & \longrightarrow & \m & \longrightarrow & \,\,\Delta_+ & \longrightarrow & 0 \\
&& \downarrow && \downarrow && \downarrow \\
0 & \longrightarrow & \,\,\nu_- & \longrightarrow & \nu & \longrightarrow & \,\,\hat{\nu}_+ & \longrightarrow & 0 \\
&& \downarrow && \downarrow && \downarrow \\
&& 0 && 0 && 0
\end{array}
 $$
where $\hat{\mathfrak{o}}_+=\mathfrak{o}/\mathfrak{o}_-$ and $\hat{\nu}_+=\Delta_+/\hat{\mathfrak{o}}_+$.

From (2) and (3) we conclude that $\mathfrak{o}_\pm\subset\Delta_\pm$ are proper subbundles, so the
rank of each of them is either 1 or 2. If the rank of $\mathfrak{o}_\e$ is 1, then the corresponding representation
of $\h\subset\mathfrak{sl}_3$ on $V_\e$ has matrix form with two rows zero, and since the isotropy representation
is faithful we get $\dim\h\leq2$. Then $\dim\g=\dim\h+\dim\mathfrak{o}\le2+5<9$.

Thus we have to consider only the case when ranks of both $\mathfrak{o}_\pm$ are 2. In this case the
matrix representation of $\h$ has one trivial row, and $\dim\h\le5$. So if $\dim\mathfrak{o}<4$ we
conclude $\dim\g<9$.
Consider the cases $4\leq\dim\mathfrak{o}\leq5$, $\dim\mathfrak{o}_\pm=2$.

A plane $\mathfrak{o}_-$ in $\Delta_-$ determines by duality a line $N_J(\we^2\mathfrak{o}_-)\subset\Delta_+$.
If this does not belong to $\mathfrak{o}_+$, then $\Delta_+$ is split into the sum
of a plane (non-trivial $\h$-module) and line (trivial $\h$-module),
whence $\dim\h\leq3$ that implies $\dim\g<9$. Thus we assume that the line
belongs to $\mathfrak{o}_+$, and so $\Delta_+$ has an invariant flag of subspaces.
Consequently $\h\subset\mathfrak{p}_{12}$ -- the Borel subalgebra in $\mathfrak{sl}_3$.
Representation on $\Delta_-$ is dual, and we obtain the matrix form of the representation $\rho:\h\to\op{End}(\m)$:
 $$
\left[
\begin{array}{c@{}c@{}c}
 \left[\begin{array}{ccc}
 a_1 & a_2 & a_3 \\
 0 & -a_1-a_5 & a_4 \\
 0 & 0 & a_5
 \end{array}\right] & \mathbf{0} \\
  \mathbf{0} & \left[\begin{array}{ccc}
 -a_1 & 0 & 0 \\
 -a_2 & a_1+a_5 & 0 \\
 -a_3 & -a_4 & -a_5
 \end{array}\right]
\end{array}\right]
 $$

\com{
\[\left[\begin{array}{ccc|ccc}
* & * & * & 0 & 0 & 0 \\
0 & * & * & 0 & 0 & 0 \\
0 & 0 & * & 0 & 0 & 0 \\
\hline
0 & 0 & 0 & * & 0 & 0 \\
0 & 0 & 0 & * & * & 0 \\
0 & 0 & 0 & * & * & *
\end{array}\right]\]
}

The first $2\times2$ sub-block in the first diagonal $3\times3$ block ($\Delta_+$)
corresponds to $\mathfrak{o}_+$, and so by property (4) the last row of this $3\times3$ block vanishes:
$a_5=0$. Similarly,
the second $2\times2$ sub-block in the second diagonal $3\times3$ block ($\Delta_-$)
corresponds to $\mathfrak{o}_-$, and so by property (4) the first row of this $3\times3$ block vanishes:
$a_1=0$. Therefore $\dim\h\le3$, and $\dim\g\leq\dim\h+5<9$.

We conclude that for $\dim\g\ge9$ the symmetry acts locally transitively
(has an open orbit) in $U_{\op{reg}}$.
This finishes the proof of Theorem \ref{Thm3}.

 \begin{rk}
Thus non-degenerate para-complex 6D manifolds with the symmetry of dimension at least 9
are locally homogeneous spaces and hence can be classified.
In particular, for $\dim\g=9$ similarly to \cite{AKW},
one can obtain a classification of such homogeneous spaces with a simple isotropy $\h$.
 \end{rk}

%%%%%%%%%%%%%%%%%%%%%%%%%%%%%%%%%%%%%%%%%%%%%%%%%%%%%%%%%%%%%%%%%%%%%%%%%%%%


\begin{thebibliography}{WW}
 \footnotesize

\bibitem{A}
D. Alekseevsky, {\it Pseudo-K\"ahler and para-K\"ahler symmetric spaces},
Handbook of Pseudo-Riemannian Geometry and Supersymmetry,
IRMA Lect. Math. Theor. Phys. {\bf 16}, Ed: V. Cort\'es , EMS, Zurich (2010)

\bibitem{AKW}
D.\,V. Alekseevsky, B.\,S. Kruglikov, H. Winther, {\it Homogeneous almost complex structures in dimension 6 with semi-simple isotropy\/}, Ann. Glob. Anal. Geom. {\bf 46}, 361–387 (2014)

\bibitem{AMT}
D.\,V.\ Alekseevsky, C.\ Medori, A.\ Tomassini, {\it Homogeneous para-K\"ahler Einstein manifolds\/},
 % Uspekhi Mat.Nauk (Russian) {\bf 64}:1, 3–50 (2009); Engl:
Russian Math. Surveys {\bf 64}:1, 1–4 (2009)

\bibitem{Br}
R.L. Bryant, {\it Conformal geometry and 3-plane fields on 6-manifolds\/},
Developments of Cartan Geometry and Related Mathematical Problems, RIMS Symposium Proceedings,
Kyoto University, vol.{\bf 150}:2, 1–15 (2006)

\bibitem{Bu}
J.-B. Butruille, {\it Classification des vari\'et\'es approximativement k\"ahleriennes homog\`enes\/},
Ann. Global Anal. Geom. {\bf 27}, no. 3, 201-225 (2005)

\bibitem{CFG}
V. Cruceanu, P. Fortuny, P.M. Gadea, {\it A survey on paracomplex geometry\/},
Rocky Mountain J. Math. {\bf 26}, 83–115 (1996)

\bibitem{Gi}
A. Guichardet, {\it Cohomologie des groupes topologiques et des algèbres de Lie\/}, Cedic/Nathan, Paris (1980)

\bibitem{Gr}
A. Gray, {\it Nearly K\"ahler manifolds\/}, J. Differential Geometry {\bf 4}, 283-309 (1970)

\bibitem{GS}
J.\,B.\ Gutowski, W.\,A.\ Sabra, {\it Para-complex geometry and gravitational instantons\/},
Classical and Quantum Gravity {\bf 30}, 195001 (2013)

\bibitem{HDKN}
Z. Hou, S. Deng, S. Kaneyuki, K. Nishiyama,
{\it Dipolarizations in semisimple Lie algebras and homogeneous para-K\"ahler manifolds\/},
J. Lie Theory {\bf 9} (1), 215-232 (1999)

\bibitem{IZ}
S. Ivanov, S. Zamkovoy, {\it ParaHermitian and paraquaternionic manifolds\/},
Diff. Geom. Appl. {\bf 23}, no.2, 205-234 (2005).

\bibitem{Ka}
I. Kath, {\it $G^*_{2(2)}$-structures on pseudo-Riemannian manifolds\/},
J. Geom. Phys. {\bf 27}, no. 3-4, 155--177 (1998)

\bibitem{K1}
B. Kruglikov, {\it Non-existence of higher-dimensional pseudoholomorphic submanifolds\/},
Manuscripta Mathematica, {\bf 111}, 51-69 (2003)

\bibitem{K2}
B. Kruglikov, {\it Symmetries of  almost complex structures and pseudoholomorphic foliations\/},
Internat. J. Math. {\bf25}:8, 1450079 (2014)

\bibitem{KL}
B. Kruglikov, V. Lychagin, {\it Geometry of Differential equations\/},
Handbook of Global Analysis, Ed. D.Krupka, D.Saunders, Elsevier, 725-772 (2008)

\bibitem{KT}
B.\ Kruglikov, D.\ The, {\it The gap phenomenon in parabolic geometry\/},
arXiv:1303.1307; J. reine angew. Math. (Crelle's Journal), DOI 10.1515/ crelle-2014-0072(2014)

\bibitem{KW} B.\ Kruglikov, H.\ Winther,
{\it Almost complex structures in 6D with non-degenerate Nijenhuis tensors and large symmetry groups},
Ann. Glob. Anal. Geom. {\bf 50}, 297-314 (2016)

\bibitem{note}
B.\ Kruglikov, H.\ Winther, {\it Reconstruction from Representations: Jacobi via Cohomology\/},
arXiv:1611.05334; to appear in Journal of Lie Theory (2017)

\bibitem{L}
P. Libermann, {\it Sur les structures presque paracomplexes\/} (French),
C. R. Acad. Sci., Paris {\bf 234}, 2517-2519 (1952)

\bibitem{M}
G.\,D.\ Mostow, {\it On maximal subgroups of real Lie groups\/}, Ann. of Math. {\bf 74}, 503–517 (1961)

\bibitem{R}
P.\,K.\ Rashevskii, {\it The scalar fields in a stratified space\/}, Trudy Sem. Vektor. Tenzor. Analizu
(Russian) {\bf 6}, 225–248 (1948)

\bibitem{Sc}
L. Sch\"afer, {\it Para-$tt^*$-bundles on the tangent bundle of an almost para-complex manifold\/},
Ann. Global Anal. Geom. {\bf 32} (2), 125-145 (2007)

\bibitem{Sp}
D.\,C.\ Spencer, {\it Overdetermined systems of linear partial
differential equations\/}, Bull.\ Amer.\ Math.\ Soc., {\bf 75}, 179--239 (1969)

\bibitem{T}
W.\,P.\ Thurston, {\it A generalization of the Reeb stability theorem\/}, Topology {\bf 13}, 347–352 (1974).

\end{thebibliography}
\end{document}